\documentclass[11pt,reqno]{amsart}
\usepackage{amsmath, amssymb, amsthm, esint, verbatim, hyperref}
\usepackage{times}
\usepackage{dsfont}
\usepackage{mathtools}
\usepackage{bbm}			

\numberwithin{equation}{section}

\hypersetup{
  pdftitle={A new $p$-harmonic map flow with Struwe monotonicity},
  pdfauthor={Erik Hupp and Michal Miskiewicz},
  pdfsubject={},
  pdfkeywords={},
  pdfpagelayout=SinglePage,
  pdfpagemode=UseOutlines,
  colorlinks,
  bookmarksopen,
  linkcolor=[rgb]{0,0,0.7},
  urlcolor=[rgb]{0,0,0.4},
  citecolor=[rgb]{0.4,0.1,0},
}

\newtheorem{theorem}{Theorem}[section]

\newtheorem{proposition}[theorem]{Proposition}

\newtheorem{lemma}[theorem]{Lemma}
\newtheorem{corollary}[theorem]{Corollary}

\theoremstyle{definition}
\newtheorem{definition}[theorem]{Definition}
\newtheorem{notation}[theorem]{Notation}
\theoremstyle{remark}
\newtheorem{remark}[theorem]{Remark}
\theoremstyle{remark}

\theoremstyle{remark}

\theoremstyle{remark}

\theoremstyle{remark}

\newcommand{\dist}{\mathrm{dist}}

\newcommand{\Lip}{\mathrm{Lip}}

\newcommand{\cA}{\mathcal{A}}

\newcommand{\cS}{\mathcal{S}}

\newcommand{\cH}{\mathcal{H}}

\newcommand{\eps}{\varepsilon}		

\newcommand{\pl}{\partial}
\newcommand{\pheq}{\phantom{=}}		

\newcommand{\ov}[1]{\overline{#1}}

\newcommand{\R}{\mathds{R}}
\renewcommand{\S}{\mathbb{S}}

\newcommand{\ps}[2]{\left\langle#1,#2\right\rangle}
\newcommand{\ton}[1]{\left(#1\right)}

\newcommand{\wto}{\rightharpoonup}
\newcommand{\B}[2]{B_{#1}\ifthenelse{\isempty{#2}}{}{\ton{#2}}}		

\newcommand{\dd}{\, d}		
\newcommand{\dv}{\operatorname{div}}
\newcommand{\Rm}{\operatorname{Rm}}
\newcommand{\Min}{\operatorname{Min}}
\newcommand{\cN}{\mathcal{N}}
\newcommand{\cM}{\mathcal{M}}
\newcommand{\ale}{\lesssim}

\newcommand{\Loc}{{\mathrm{Loc}}}

\newcommand{\supp}{\operatorname{supp}}

\title{A new $p$-harmonic map flow with Struwe monotonicity}

\makeatletter
\@namedef{subjclassname@2020}{\textup{2020} Mathematics Subject Classification}
\makeatother


\author{Erik Hupp}
\address[Erik Hupp]{
Department of Mathematics, 
Northwestern University, 
2033 Sheridan Road, 
Evanston IL 60208, USA}
\email{ehupp@math.northwestern.edu}

\author{Micha\l{} Mi\'{s}kiewicz}
\address[Micha\l{} Mi\'{s}kiewicz]{
Institute of Mathematics, 
University of Warsaw, 
Banacha 2, 
02-097 Warszawa, Poland}
\email{m.miskiewicz@mimuw.edu.pl}

\thanks{}
\date{}

\begin{document}

\begin{abstract}
We construct and analyze solutions to a regularized homogeneous $p$-harmonic map flow equation for general $p \geq 2$. The homogeneous version of the problem is new and features a monotonicity formula extending the one found by Struwe \cite{Str88} for $p = 2$; such a formula is not available for the nonhomogeneous equation. The construction itself is via a Ginzburg-Landau-type approximation \`a la Chen-Struwe \cite{CheStr89}, employing tools such as a Bochner-type formula and an $\eps$-regularity theorem. We similarly obtain strong subsequential convergence of the approximations away from a concentration set with parabolic codimension at least $p$. However, the quasilinear and non-divergence nature of the equation presents new obstacles that do not appear in the classical case $p = 2$, namely uniform-time existence for the approximating problem, and thus our basic existence result is stated conditionally. 
\end{abstract}

\maketitle
\vspace{-3mm}
{\small \tableofcontents}

\section{Introduction} \label{sec:Intro}

\subsection{Motivation and previous work}

One of central objects in differential geometry are harmonic maps $u \colon \cM^n \to \cN^m$, which are defined as critical points of the Dirichlet energy $E_2(u) = \int_\cM |\nabla u|^2$. The evolutional approach to studying these maps, based on the gradient flow of the energy called the \emph{harmonic map flow}, proved to be very fruitful. First considered by Eells and Sampson \cite{EelSam64} for a negatively curved target $\cN$, it was later developed for arbitrary targets, first in the conformal dimension $n = 2$ \cite{Str85} and then in the general case \cite{CheStr89}. With $\cN$ assumed to be embedded into some Euclidean space $\R^d$, the solution of the harmonic map flow can be described as solving the system 
\begin{equation}
\label{eq:classical-HMF}
\pl_t u - \Delta u = A_u(\nabla u, \nabla u)\,,
\end{equation}
where $A$ denotes the second fundamental form of $\cN \subseteq \R^d$. It is known that there exist global-in-time weak solutions; their possible singularities (forming a set of parabolic Hausdorff dimension at most $n$) are by now well understood. Certain special cases -- target $\cN$ having nonpositive sectional curvature; target $\cN$ being the round sphere $\S^m$; small initial energy $E_2(u_0)$; conformal dimension $n = 2$ -- allow for simplifications and/or stronger results. 

A natural extension of this theory is the study of $p$-harmonic maps, which in turn are critical points of the $p$-energy $E_p(u) = \int_\cM |\nabla u|^p$. As for their evolutional counterpart, there is extensive literature focusing on the \emph{(nonhomogeneous) $p$-harmonic map flow} system 
\begin{equation}
\label{eq:nonhomogeneous-p-flow}
\pl_t u - \Delta_p u = |\nabla u|^{p-2} A_u(\nabla u, \nabla u)\,,
\end{equation}
where now $\Delta_p u$ is the $p$-Laplace operator $\dv(|\nabla u|^{p-2} \nabla u)$. 

Hungerb\"uhler constructed global weak solutions in the conformal case $p = \dim \cM$ together with a~uniqueness result \cite[Th.~4.9,~5.1]{Hun94}, as well as for the flow into spheres \cite{CheHonHun94} with Chen and Hong (see also \cite[Th.~2.4]{Hun94}, and \cite{Hun96} for the more general case of homogeneous spaces). Similar results due to Fardoun and Regbaoui are known in the case of small initial energy $E_p(u_0)$ \cite{FarReg03} and nonpositive section curvature of $\cN$ \cite{FarReg02}. Recently, Misawa obtained global solutions also in the general case \cite{Mis19} -- see Remark \ref{rmk:Misawa-rmk} for a comparison between this last paper's estimates and ours. 

\medskip

However, there is an ambiguity in forming the gradient flow of the $p$-energy functional for $p \neq 2$, and the nonhomogeneous formulation in \eqref{eq:nonhomogeneous-p-flow} is not the only possible choice. Instead, one could consider the $p$-energy as $2$-energy: $E_p(u) = \int_\cM |\nabla u|^{2} |\nabla u|^{p-2}$, where $|\nabla u|^{p-2}$ is treated as a weight. Formally computing the gradient of the functional $E_p$ at $u$ with respect to the inner product $\int_\cM \ps{\cdot}{\cdot}|\nabla u|^{p - 2}$ leads to the following gradient flow: 
\begin{equation}
\label{eq:unreg-p-HMF}
\pl_t u - |\nabla u|^{2-p} \Delta_p u = A_u(\nabla u,\nabla u)\,,
\end{equation}
which we call the \emph{homogeneous $p$-harmonic map flow}. Since the operator $|\nabla u|^{2-p} \Delta_p u$ is ill-defined at points where the gradient vanishes, in this paper we focus on the regularized version of this system: 
\begin{equation}
\label{eq:reg-p-HMF-intro}
\pl_t u - (|\nabla u|^2 + \delta^2)^\frac{2 - p}{2} \dv( (|\nabla u|^2 + \delta^2)^\frac{p - 2}{2} \nabla u) = A_u(\nabla u,\nabla u)\,,
\end{equation}
in which $|\nabla u|^2$ is replaced by $|\nabla u|^2 + \delta^2$, with some $\delta > 0$ fixed. Whenever it is possible, we derive estimates with constants independent of $\delta$.

The homogeneous formulation of the $p$-harmonic map flow \eqref{eq:unreg-p-HMF} has a number of analytic advantages over its nonhomogeneous counterpart, which we describe in detail in the next subsection. But first, let us note that the homogeneous $p$-Laplace operator $|\nabla u|^{2-p} \Delta_p u$ (for scalar functions) appears naturally in game theory as describing the value function of tug-of-war games with noise \cite{PerShe08}, hence it is also called the ``game-theoretic'' or ``normalized'' $p$-Laplace operator in the literature. 

The flow \eqref{eq:unreg-p-HMF} and its regularization \eqref{eq:reg-p-HMF-intro} have precedent in the scalar case, in which the target manifold is simply $\cN = \R$. Here, Banerjee and Garofalo \cite{BanGar13} prove existence and regularity, with a view toward $p \searrow 1$; the authors also derive a Struwe-type monotonicity formula. This scalar setting has been attracting an increasing amount of attention, for instance in the work of \cite{JinSil17}, \cite{AttPar18}, \cite{ParVaz20}, and others.

The fundamental obstacle in our context is that \eqref{eq:unreg-p-HMF} and \eqref{eq:reg-p-HMF-intro} are \emph{systems} of equations which cannot be reduced to the scalar case. Hence, most of the methods used in \cite{BanGar13} and other works are not applicable. Although a viscosity-type interpretation is possible for systems (see \cite{Kat22}), it is designed for regular solutions and may not be well-adapted to analyzing the singularities expected in the geometric case. For this and other technical reasons, we have stopped short of sending $\delta \to 0$, although we note that taking $p = 2$ recovers the classical system, even for $\delta>0$. 

\subsection{Main results}

Arguably the main advantage of the homogeneous $p$-harmonic map flow \eqref{eq:unreg-p-HMF} is that it supports a Struwe-type monotone quantity strictly analogous to the $p = 2$ case: by a formal computation, 
\begin{equation}
\Phi_{(t,x)} (s; u) := s^{p/2} \int_{\R^n} |\nabla u_{t - s}|^p (x-y) \rho_s(x) \dd x
\end{equation}
is a nondecreasing function of $s > 0$ (here, $\rho_s$ is the standard heat kernel in $\R^n$). An analogous monotonicity formula holds for the regularized flow \eqref{eq:reg-p-HMF-intro} (see Lemma \ref{lem:delta-monotonicity-formula}) and even for its Ginzburg-Landau approximation (see Theorem \ref{th:monotonicity-formula}) which we now mention briefly. 

Following the general strategy of \cite{CheStr89}, we consider a Ginzburg-Landau approximation of the regularized flow \eqref{eq:reg-p-HMF-intro}. The idea is to relax the geometric constraint $u_t(x) \in \cN$ by considering maps into the ambient space $\R^d$ and at the same time adding a penalizing term $K^2 F(u)$ to the energy functional; thus, the nonlinearity on the right hand side is replaced by a lower order term $- \frac 12 K^2 \nabla F(u)$. Here, $K > 0$ is a large parameter eventually taken to the limit $K \to \infty$, and $F \colon \R^d \to \R$ is a smooth truncation of $\dist^2(x,\cN)$ (see Section \ref{sec:G-L} for details). Most of our results apply equally to the $\delta$-flow \eqref{eq:reg-p-HMF-intro} and to the $(\delta,K)$-flow -- its Ginzburg-Landau approximation. 

Regular solutions of the $(\delta,K)$-flow satisfy a Bochner-type formula 
\begin{equation}
(\pl_t-\cA) |\nabla u|_{\delta,K}^p
\ale_\cN |\nabla u|_{\delta,K}^{p+2}\,,
\end{equation}
where $|\nabla u|^2_{\delta,K}$ is a shorthand for $|\nabla u|^2 + \delta^2 + K^2 F(u)$ and $\cA$ 
is a suitably chosen elliptic operator (see Theorem \ref{th:Bochner}). A similar formula holds for the $\delta$-flow (see Lemma \ref{lem:delta-Bochner-explicit}). 

With these two ingredients in place, we were able to derive an $\eps$-regularity result, which we again state for the $(\delta,K)$-flow (see Theorem \ref{th:eps-regularity}). That is, if the monotone quantity $\Phi^{\delta,K}_{(0,0)}(1)$ -- defined as above, but with $|\nabla u|$ replaced by $|\nabla u|_{\delta,K}$ -- is smaller than $\eps_0(n,\cN)$, then $|\nabla u|_{\delta,K}$ is uniformly bounded in a small cylinder. Crucially, the constant $\eps_0$ \emph{does not depend} on $\delta$ or $K$. This is again a clear analogue of Chen and Struwe's $\eps$-regularity for the Ginzburg-Landau approximation in the case $p=2$. However, the proof of weighted $L^2$ Hessian estimates within $\eps$-regularity regions (Theorem \ref{th:eps-reg-W22}) requires new ideas, including a~trick using a stationary-like equation, given that the integration-by-parts method used for $p = 2$ is no longer available in this non-divergence setting. 

Finally, we study the $K \to \infty$ limit of the approximating solutions, obtaining strong convergence away from a concentration set with bounded $(n + 2 - p)$-Minkowski content. Moreover, we obtain weak $W^{1,2}_{\Loc}$ convergence to a weak solution of \eqref{eq:reg-p-HMF-intro} over the whole space-time domain (Theorem \ref{th:HMF-distr}).

\medskip

As already noted, the homogeneous flow is a quasilinear parabolic system in nondivergence form. The general theory for such systems does not provide global-in-time existence results, even for uniformly strongly elliptic operators (in the sense of Legendre) such as the $\delta$-regularized homogeneous $p$-Laplacian. Hence, even for the simplest case of the $\delta$-flow \eqref{eq:reg-p-HMF-intro} with $\cN = \R^d$, i.e. with zero right-hand side, the existence of global solutions is an open problem. We direct the reader to Remark \ref{rmk:flow-existence} for a more in-depth discussion. For these reasons, our construction is conditional -- we construct the solution of the geometric flow assuming that its (much better behaved) Ginzburg-Landau approximation can be solved. 

\subsection{Outline}

Sections \ref{sec:G-L}, \ref{sec:Eps-Reg}, and \ref{sec:K-Limit} follow the general strategy used by Chen and Struwe \cite{CheStr89} in the case $p = 2$, although technical obstacles for $p > 2$ require new ideas in certain critical parts of the argument. More precisely, in Section \ref{sec:G-L} we introduce the Ginzburg-Landau approximation to the $\delta$-flow \eqref{eq:reg-p-HMF-intro}. We discuss and derive some preliminary results for its solutions, including the $p > 2$ analog of Struwe's celebrated monotonicity formula. Then in Section \ref{sec:Eps-Reg} we derive a new Bochner inequality, which leads to an $\eps$-regularity theorem. We  also establish weighted $L^2$ Hessian estimates in $\eps$-regularity regions. Section \ref{sec:K-Limit} is devoted to studying the $K \to \infty$ limit of Ginzburg-Landau approximating solutions -- first in $\eps$-regularity regions and then globally. 

Finally, Section \ref{sec:Apps} aims to reproduce the results known for $p = 2$ in certain interesting special cases. It is shown that the $\delta$-flow \eqref{eq:reg-p-HMF-intro} is better behaved under the geometric assumption that the target manifold is negatively curved, as well as under the analytic assumption that the initial energy is small. 

\subsection{Notations and conventions}

Let us fix the following, possibly nonstandard, general notation that will be used without further comment throughout the paper.

\begin{notation}
\leavevmode
\begin{itemize}

\item The domain of the flow $u$ will be $(0,\infty) \times \R^n$ and by default, the time variable $t$ is the first coordinate. We shall use subscript to denote evaluation, so that $u_t$ will denote the map $x \mapsto u_t(x)$; the time derivative will be denoted as $\pl_t u$. 

\item The standard heat kernel on Euclidean space evaluated at a point $(t,x)\in (0,\infty) \times \R^n$ is denoted by $\rho_{t}(x) := (4\pi t)^{-n/2}e^{-\frac{|x|^2}{4t}}$. 

\item The solution for the heat equation on $(0,\infty)\times \R^n$ with initial data (function or tempered distribution) $f \colon \R^n \to \R$ and evaluated at $(t,x)$ is denoted by
\begin{equation}
H_t[f](x) := \int \rho_t(x - y)f(y) \dd y\,,
\end{equation}
with $H_0 [f] = f$. 

\item A (homogeneous) parabolic cylinder of radius $r$ centered at the point $(t,x) =: z \in \R\times\R^n$ is denoted by
\begin{equation}
P_r(z) := \{(s,y) \in \R \times \R^n \mid 0 < t - s < r^2,\, |x - y| < r\}\,,
\end{equation}
and the parabolic (metric) ball is denoted by 
\begin{equation}
Q_r(z) := \{(s,y) \in \R \times \R^n \mid |t - s| < r^2,\, |x - y| < r\}\,.
\end{equation}

\end{itemize}
\end{notation}

\subsection{Acknowledgments}

The research of the first named author was partially supported by NSF Grant DMS-1502632 RTG: Analysis on Manifolds, while the research of the second named author was supported by the NCN Sonatina grant no. 2020/36/C/ST1/00050. Both authors would like to express their gratitude to prof. Aaron Naber for suggesting this research topic as well as for many helpful discussions.

\section{The Ginzburg-Landau approximation} \label{sec:G-L}

\subsection{Definitions and notation}

The regularization we introduce has two aspects. First, the homogeneous $p$-Laplace operator $|\nabla u|^{2-p} \dv(|\nabla u|^{p-2} \nabla u)$ can be approximated by considering the $\delta$-energy 
\begin{equation}
E_{\delta}(u) := \frac 1p \int \left( |\nabla u|^2 + \delta^2 \right)^{p/2}\,,
\end{equation}
where $|\nabla u|^2$ was replaced by $\delta^2+|\nabla u|^2$. The corresponding operator is 
\begin{equation}
(\delta^2 + |\nabla u|^2)^{\frac{2-p}{2}} \dv\left( (\delta^2 + |\nabla u|^2)^{\frac{p-2}{2}} \nabla u \right)\,.
\end{equation}
In fact, for this paper we restrict our attention to operators regularized in this way, with $\delta > 0$ fixed throughout.

Second, even with this regularized operator, constructing the corresponding flow with values in $\cN$ is non-trivial. In analogy to the case $p = 2$ considered by Chen-Struwe \cite{CheStr89}, we introduce a Ginzburg-Landau type approximation -- this means that the constraint $u(x) \in \cN$ is replaced by adding a penalizing term $F$ with a large weight. That is, we choose a function $F(p)$ that behaves like $\dist^2(p,\cN)$ and consider the operator 
\begin{equation}
(\delta^2 + K^2 F(u) + |\nabla u|^2)^{\frac{2-p}{2}} \dv\left( (\delta^2 + K^2 F(u) + |\nabla u|^2)^{\frac{p-2}{2}} \nabla u \right) - \frac 12 K^2 (\nabla F) (u)\,.
\end{equation}
This precise form is justified by considering the energy functional, which we define as

\begin{definition}[Penalized Energy]
Choose $r_0 = r_0(\cN) > 0$ so that nearest-point projection $\pi_{\cN}: B_{r_0}(\cN) \subset \R^d \to \cN$ is well-defined and smooth. We let 
\begin{equation}
\label{eq:energy-delta-K}
E_{\delta,K}(u) := \frac 1p \int \left( |\nabla u(x)|^2 + F_K(u(x)) + \delta^2 \right)^{p/2}\,,
\end{equation}
where
\begin{align}
F_K(y) &:= K^2 F(y) := K^2\chi(\dist^2(y,\cN))\,, \\
\chi(t)&:= \begin{cases}
t,& t \leq r_0^2 \,, \\
(2r_0)^2,& t \geq (2r_0)^2 \,, \\
\text{smooth, increasing},& r_0^2 \leq t \leq (2r_0)^2 \,.
\end{cases}
\end{align}
\end{definition}

Note that subscripts may be dropped in later sections when there is no cause for confusion about which parameters $\delta,K > 0$ are being considered.

\begin{notation}
In order to streamline computations, we introduce some shorthands related to the penalized energy density and related quantities:
\begin{equation}
e_{\delta,K}(u) := \delta^2 + F_K(u) + |\nabla u|^2\,,
\qquad 
|\nabla u|_{\delta,K} := \sqrt{e_{\delta,K}(u)} = (\delta^2 + F_K(u) + |\nabla u|^2)^{1/2}\,,
\end{equation}
which correspond to $e(u) := e_{0,0}(u) = |\nabla u|^2$ and $|\nabla u|_{0,0} = |\nabla u|$ in the case of the classical harmonic map flow. Similarly, one can consider $e_\delta(u) := e_{\delta,0}(u)$ and $|\nabla u|_\delta := |\nabla u|_{\delta,0}$. 
\end{notation}

With this notation, the energy $E_{\delta,K}$ has the form 
\begin{equation}
E_{\delta,K}(u) 
= \int (e_{K,\delta}(u))^{p/2} 
= \int |\nabla u|_{\delta,K}^p\,,
\end{equation}
and computing an appropriately scaled Euler-Lagrange equation for $E_{\delta,K}$ yields the evolutionary equation, which we may refer to as $(\delta,K)$-flow: 
\begin{equation}
\label{eq:flow-delta-K}
\partial_t u - |\nabla u|_{\delta,K}^{2-p} \dv \left( |\nabla u|_{\delta,K}^{p-2} \nabla u \right) 
= - \frac 12 (\nabla F_K)(u)\,.
\end{equation}
Similarly, let us also record the $\delta$-regularized $p$-harmonic map flow equation, or more succinctly $\delta$-flow: 
\begin{equation}
\label{eq:reg-flow}
\pl_t u - |\nabla u|_\delta^{2 - p}\dv(|\nabla u|_\delta^{p-2} \nabla u) = A_u(\nabla u, \nabla u)\,,
\end{equation}
introduced earlier as \eqref{eq:reg-p-HMF-intro}. 

\subsection{Basic Properties of the $(\delta,K)$-flow}

Let us begin with some terminology from the theory of elliptic systems (see \cite[Def.~3.36]{GiaMar12}):

\begin{definition}
A second order (linear) differential operator on $C^\infty(\Omega\subset\R^n,\R^d)$ with principal symbol $a_{\alpha,\beta}^{k,i}(x)\xi^\alpha\xi^\beta \eta_k \eta_i$ ($x\in\Omega$, $\xi \in \R^n$, $\eta \in \R^d$) is elliptic in the sense of Legendre at $x$ if
\begin{equation}
\lambda|h|^2\leq \sum a_{\alpha,\beta}^{k,i}(x)h_\alpha^kh_\beta^i \leq \Lambda |h|^2
\end{equation}
for some $0 < \lambda \leq \Lambda$ and all $d\times n$ matrices $h$, with $|h|$ denoting the Hilbert-Schmidt norm. The ellipticity is uniform if $\lambda,\Lambda$ can be chosen independently of $x\in\Omega$.
\end{definition}

We will now verify that the operator in \eqref{eq:flow-delta-K} is elliptic. To this end, let us write it in the form
\begin{equation}
\label{eq:coefficients-of-operator}
|\nabla u|_{\delta,K}^{2-p} \dv \left( |\nabla u|_{\delta,K}^{p-2} \nabla u^k \right) - \frac 12 (\nabla F_K^k)(u) 
= \sum_{\alpha,\beta,i} a_{\alpha,\beta}^{k,i}(u, \nabla u) \cdot \pl_{\alpha\beta} u^i + b^k(u,\nabla u)\,.
\end{equation}

\begin{proposition}
The operator in \eqref{eq:coefficients-of-operator} above is well-defined, with smooth $a_{\alpha,\beta}^{k,i}(\cdot,\cdot)$, $b^{k}(\cdot,\cdot)$. The matrix $a_{\alpha,\beta}^{k,i}(\cdot,\cdot)$ is elliptic in the sense of Legendre, uniformly in $\delta$ and $K$. Moreover, the lower order term $b^k(\cdot,\cdot)$ is bounded by a constant times $K^2$. 
\end{proposition}

\begin{proof}
Since $|\nabla u|_{\delta,K}$ is bounded below by $\delta$, one can simply evaluate the divergence in \eqref{eq:coefficients-of-operator} to see 
that we can choose the coefficients $a$ and $b$ as follows: 
\begin{align}
\label{eq:operator-coefficients}
a_{\alpha,\beta}^{k,i} 
& = \delta_{\alpha,\beta} \delta_{k,i} + (p-2) \frac{\pl_\alpha u^k}{|\nabla u|_{\delta,K}} \cdot \frac{\pl_\beta u^i}{|\nabla u|_{\delta,K}}\,, \\
b^k & = K^2 \cdot \left( - \frac 12 \pl_k F (u) + \frac{p-2}{2} \sum_{\alpha,i} \pl_i F(u) \, \frac{\pl_\alpha u^i}{|\nabla u|_{\delta,K}} \cdot \frac{\pl_\alpha u^k}{|\nabla u|_{\delta,K}} \right)\,.
\end{align}
As all the matrices $\pl_\alpha u^k / |\nabla u|_{\delta,K}$ have norm bounded by $1$, $b$ is indeed bounded by a constant (depending on $\cN$) times $K^2$. Now, take any matrix $h_\alpha^i$ and consider the sum 
\begin{equation}
\label{eq:Legendre-condition}
\sum_{\alpha,\beta,i,k} a_{\alpha,\beta}^{k,i} h_\alpha^k h_\beta^i 
= \sum_{\alpha,k} |h_\alpha^k|^2 + (p-2) \left| \sum_{\alpha,k} \frac{\pl_\alpha u^k}{|\nabla u|_{\delta,K}} \cdot h_\alpha^k \right|^2\,.
\end{equation}
This is bounded below by $|h|^2$ (using the Hilbert-Schmidt norm, as usual) and above by $(p-1) |h|^2$, thus implying uniform monotonicity. 
\end{proof}

Moreover, we have a parabolic scaling symmetry for solutions to \eqref{eq:flow-delta-K}.

\begin{proposition}[parabolic scaling]
\label{prop:para-scal}
Let $u$ be a solution of \eqref{eq:flow-delta-K} with parameters $\delta,K$. Then the rescaled flow $u^{(\lambda)}(t,x) := u(\lambda^2 t, \lambda x)$ is a solution of \eqref{eq:flow-delta-K} with parameters $\lambda \delta,\lambda K$. 
\end{proposition}

This is in contrast with the non-homogeneous $p$-flow $\pl_t u - \dv(|\nabla u|^{p-2} \nabla u) = |\nabla u|^{p-2} A_u(\nabla u, \nabla u)$, where the equation is invariant with respect to a different scaling: $u^{(\lambda)}(t,x) := u(\lambda^p t, \lambda x)$.

Finally, we record two energy inequalities which generalize estimates known in the classical case $p = 2$. We note however that in the case $p \neq 2$, the resulting integral is weighted by a nonzero power of $|\nabla u|_{\delta,K}$.

\begin{lemma}[Global energy inequality]
\label{lem:energy-inequality}
Let $\cM$ be a closed manifold and let $u \colon (-\eps,\eps) \times \cM \to \R^d$ be a smooth solution of \eqref{eq:flow-delta-K} with parameters $\delta,K$. Then the global energy $E_{\delta,K}(u_t) := \frac 1p \int_{\cM} |\nabla u|^p_{\delta,K}$ is a~non-increasing function of $t$, and more precisely, 
\begin{equation}
\pl_t \left( E_{\delta,K}(u_t) \right)
= - \int_{\cM} |\nabla u|^{p-2}_{\delta,K} |\pl_t u|^2\,.
\end{equation}
\end{lemma}

\begin{proof}
Note that $\pl_t |\nabla u|^2_{\delta,K} = 2 \nabla u \cdot \nabla \pl_t u + \nabla F_K(u) \cdot \pl_t u$, and hence differentiating $\frac 1p |\nabla u|^p_{\delta,K}$ and then integrating in space gives: 
\begin{align}
\pl_t \left( \frac 1p |\nabla u|^p_{\delta,K} \right) 
& = |\nabla u|^{p-2}_{\delta,K} \nabla u \cdot \nabla \pl_t u + \frac 12 |\nabla u|^{p-2}_{\delta,K}  \nabla F_K(u) \cdot \pl_t u\,, \\
\pl_t E_{\delta,K}(u)
& = \pl_t \left( \frac 1p \int |\nabla u|^p_{\delta,K} \right) \notag\\
& = \int - \dv (|\nabla u|^{p-2}_{\delta,K} \nabla u) \cdot \pl_t u + \frac 12 |\nabla u|^{p-2}_{\delta,K} \nabla F_K(u) \cdot \pl_t u \notag\\
& = - \int |\nabla u|^{p-2}_{\delta,K} \left( |\nabla u|^{2-p}_{\delta,K} \dv (|\nabla u|^{p-2}_{\delta,K} \nabla u) - \frac 12 \nabla F_K(u) \right) \cdot \pl_t u\notag\\
& = - \int |\nabla u|^{p-2}_{\delta,K} |\pl_t u|^2\,. 
\end{align}
\end{proof}

\begin{remark}
Due to the positive term $\delta > 0$, the above Lemma \ref{lem:energy-inequality} does not hold if $\cM = \R^n$, simply because the energy $E_{\delta,K}$ is infinite. To resolve this issue, one can consider a modified version, e.g. with the integrand $|\nabla u|_{\delta,K}^p - \delta^p$. However, it is the local version below that we use in the sequel (except for Theorem~\ref{th:nonpositive-curvature-limit}), so we have restricted our attention to flows on a closed manifold to minimize unilluminating technical details. 
\end{remark}

We will make extensive use of a localized variant of Lemma \ref{lem:energy-inequality}, which we package as follows

\begin{lemma}[Local energy inequality]
\label{lem:loc-energy-ineq}
Suppose that $u$ is a smooth solution of the $(\delta,K)$-flow \eqref{eq:flow-delta-K} on $P_2$. Then
\begin{equation}
\int_{P_1} |\nabla u|^{p-2}_{\delta,K}|\partial_t u|^2 
\leq C(p,n) \int_{P_2} |\nabla u|^{p}_{\delta,K}\,.
\end{equation}
\end{lemma}

\begin{proof}
Let $\chi \in C_c^\infty((-\infty,0] \times \R^n)$ be a smooth cutoff satisfying $\mathbbm{1}_{P_1} \le \chi \le \mathbbm{1}_{P_2}$ with derivatives $\nabla \chi$, $\pl_t \chi$ bounded by dimensional constants (notice that $\chi$ does not vanish at time $t=0$). Then we proceed with the previous energy inequality computation, but this time with error terms due to the cutoff: 
\begin{align}
\pl_t \left( \frac 1p \int_{\R^n} |\nabla u|^p_{\delta,K} \chi^2 \right)
& = - \int |\nabla u|^{p-2}_{\delta,K} |\pl_t u|^2 \chi^2 \notag\\
& \pheq + \frac 1p \int |\nabla u|^p_{\delta,K} \pl_t (\chi^2) - \int |\nabla u|^{p-2}_{\delta,K} \nabla u \, \pl_t u \, \nabla(\chi^2) \notag\\
& \le - \frac 12 \int |\nabla u|^{p-2}_{\delta,K} |\pl_t u|^2 \chi^2 
+ C(p,n) \int |\nabla u|^{p}_{\delta,K} \left( |\nabla \chi|^2 + |\pl_t \chi| \right)\,,
\end{align}
using Peter-Paul for the final inequality. We now integrate this inequality in time, keeping in mind that the integral over the time slice $t = -4$ vanishes: 
\begin{align}
\label{eq:local-energy-ineq-main-eq}
\frac 1p \int_{B_1} |\nabla u_0|^p_{\delta,K} + \frac 12 \int_{P_1} |\nabla u|^{p-2}_{\delta,K} |\pl_t u|^2 
& \le \frac 1p \int_{\R^n} |\nabla u_0|^p_{\delta,K} \chi_0^2 
+ \frac 12 \int_{(-4,0) \times \R^n} |\nabla u|^{p-2}_{\delta,K} |\pl_t u|^2 \chi^2 \notag\\
& \le C(p,n) \int_{(-4,0) \times \R^n} |\nabla u|^{p}_{\delta,K} \left( |\nabla \chi|^2 + |\pl_t \chi| \right) \notag\\
& \le C(p,n) \int_{P_2} |\nabla u|^{p}_{\delta,K}\,,
\end{align}
recalling the bounds on derivatives of $\chi$.
\end{proof}

We close this subsection with a remark about the difficulty of obtaining long-time solutions to the $(\delta,K)$-flow:

\begin{remark}\label{rmk:flow-existence}
The $(\delta,K)$-flow \eqref{eq:flow-delta-K} for fixed $\delta > 0$ and $K < \infty$ can be categorized as a quasilinear, uniformly parabolic, non-divergence form system, with principal symbol depending on the spatial derivative $\nabla u$. Using the Bochner formula in the upcoming Theorem \ref{th:Bochner}, along with the maximum principle, we may also assume $K$-dependent bounds on the gradient: 
\begin{equation}
\sup_{[0, T) \times \R^n}|\nabla u|_{\delta,K}
\ale_{\cN,K,p, T}\sup_{\R^n}|\nabla u_0|_{\delta,K} < \infty \,.
\end{equation}
These facts taken together, however, are not enough to conclude smoothness and long-time existence from the existing theory. Let us note that it would suffice to obtain $[\nabla u_t]_{C^{\alpha}(\R^n)} < C(T) < \infty$ for all $0 < t \leq T < \infty$, e.g. by \cite[Prop.~8.2, Pg.~389]{Tay11}. Let us also list a few ``near misses'' that do not quite apply in this context:
\begin{itemize}
\item Scalar quasilinear theory: The usual techniques used to prove existence for the $\delta$-regularized homogeneous $p$-flow in the scalar case, and quasilinear uniformly parabolic non-divergence form scalar equations in general as in \cite[Ch.~VI.1]{LadSolUra68}, do not carry over to systems in an obvious way.
\item Non-divergence systems with special form: For systems whose principal symbol is decoupled, one can usually apply scalar methods. For systems whose principal symbol depends on $u$ but not on $\nabla u$, one can apply e.g. \cite[Ch.~VII.5-7]{LadSolUra68}. Unfortunately, the $(\delta,K)$-flow has neither of these special forms.
\item Systems in divergence form: The non-homogeneous $p$-harmonic map flow into flat $\R^d$ enjoys higher regularity, see e.g. \cite{Wie86}. These methods do not apply in an obvious way to the non-divergence form $(\delta,K)$-flow.
\end{itemize}
We stress that the gap between boundedness and H\"older regularity for the gradient is not related to the geometric aspect of the $(\delta, K)$-flow, and indeed remains unresolved even for the $\delta$-flow \eqref{eq:reg-flow} into flat $\R^n$.
\end{remark}

\subsection{The monotonicity formula}

One of the main attributes of the $(\delta,K)$-flow that is not shared by its non-homogeneous counterpart is an analogue to Struwe's monotonicity formula for $p = 2$, see \cite[Prop.~3.3]{Str88}. 

\begin{definition}
Given parameters $\delta,K$, map $u \in L^\infty([0,T); W^{1,p}(\R^n,\R^d))$, point in spacetime $(t,x) \in [0,T)\times \R^n$, and $0 < s \le t$, we define the quantity 
\begin{equation}
\Phi^{\delta,K}_{(t,x)} (s; u) := s^{p/2} H_s |\nabla u_{t - s}|_{\delta,K}^p (x)\,.
\end{equation}
We will drop $\delta$, $K$, $u$ and write $\Phi_{(t,x)} (s)$, whenever it does not lead to confusion. 
\end{definition}

\begin{remark}
One can check that $\Phi$ enjoys the following invariance with respect to parabolic scaling: 
\begin{equation}
u^{(\lambda)}(t,x) = u(\lambda^2 t, \lambda x)
\quad \Longrightarrow \quad 
\Phi_{(0,0)}^{\delta,K}(\lambda^2 s; u) = \Phi_{(0,0)}^{\lambda \delta, \lambda K} \big( s; u^{(\lambda)} \big)\,.
\end{equation}
\end{remark}

\begin{remark}
The following monotonicity formula for $\Phi$ appears in the paper by Banerjee and Garofalo \cite{BanGar13} in the context of the scalar homogeneous $p$-flow. 
\end{remark}

\begin{theorem}[Monotonicity formula]
\label{th:monotonicity-formula}
If $u \in L^\infty([0,T); W^{1,p}(\R^n,\R^d))$ is a smooth solution of the $(\delta,K)$-flow \eqref{eq:flow-delta-K}, then $s \mapsto \Phi^{\delta,K}_{(t,x)}(s)$ is a non-decreasing function. More precisely, its derivative is exactly 
\begin{equation}
\pl_s \Phi_{(t,x)}^{\delta,K}(s) 
= \frac{p}{2} s^{\frac{p-4}{2}} H_s \left( |\nabla u_{t-s}|_{\delta, K}^{p-2} \left( \frac 12 |2 s \pl_t u_{t-s} - (y-x) \nabla u_{t-s}|^2 
+ s \delta^2 + s F_K(u_{t-s}) \right) \right) (x)\,.
\end{equation}
\end{theorem}

\begin{proof}
Without loss of generality, we can take $(t,x) = (0,0)$ (with $u$ defined for negative time). For simplicity, we will also consider the derivative of $\Phi_{(0,0)}(s)$ at $s = 1$; a simple scaling argument will then establish the general case. To simplify even more, we substitute $s = \lambda^2$ and consider $\Phi_{(0,0)}^{\delta,K}(\lambda^2; u) = \Phi_{(0,0)}^{\lambda \delta, \lambda K}(1; u^{(\lambda)})$, where $u^{(\lambda)}$ is the parabolic rescaling of $u$ as in Proposition \ref{prop:para-scal}. 

In result: 
\begin{equation}
\pl_\lambda \Phi_{(0,0)}^{\delta,K}(\lambda^2; u) 
= \pl_\lambda \Phi_{(0,0)}^{\lambda \delta, \lambda K}(1; u^{(\lambda)}) 
= \int_{\R^n} \pl_\lambda \left( |\nabla u^{(\lambda)}_{-1}|_{\lambda \delta, \lambda K}^p \right) \cdot \rho_1 \dd x\,. 
\end{equation}
Here the $\lambda$-derivative involves the parameters $\lambda \delta$, $\lambda K$ as well, so 
\begin{align}
\pl_\lambda \left( |\nabla u^{(\lambda)}_{-1}|_{\lambda \delta, \lambda K}^2 \right) 
& = \pl_\lambda \left( \big| \nabla u^{(\lambda)}_{-1} \big|^2 + \lambda^2 \delta^2 + \lambda^2 K^2 F \big( u^{(\lambda)}_{-1} \big) \right) \notag\\
& = 2 \nabla u^{(\lambda)}_{-1}\cdot \nabla \pl_\lambda u^{(\lambda)}_{-1} + \lambda^2 K^2 \nabla F(u^{(\lambda)}_{-1}) \cdot \pl_\lambda u^{(\lambda)}_{-1} + 2 \lambda \delta^2 + 2 \lambda K^2 F(u^{(\lambda)}_{-1})\,,
\end{align}
and the expression for the $p$-th power is similar. The last two terms are positive, so we can just record their contribution to the final result: 
\begin{equation}
\label{eq:additional-contribution}
\lambda \cdot p \int_{\R^n} |\nabla u^{(\lambda)}_{-1}|_{\lambda \delta, \lambda K}^{p-2} \left( \delta^2 + F_K(u^{(\lambda)}_{-1})\right) \cdot \rho_1 \dd x\,.
\end{equation}
For the first two terms, we integrate by parts to get 
\begin{gather}
p \int_{\R^n} |\nabla u^{(\lambda)}_{-1}|_{\lambda \delta, \lambda K}^{p-2} \left( \nabla u^{(\lambda)}_{-1}\cdot \nabla \pl_\lambda u^{(\lambda)}_{-1} + \frac{\lambda^2}{2} \nabla F_K(u^{(\lambda)}_{-1})\cdot \pl_\lambda u^{(\lambda)}_{-1} \right) \cdot \rho_1 \dd x \notag\\
= p \int_{\R^n} \left( -\dv \left( |\nabla u^{(\lambda)}_{-1}|_{\lambda \delta, \lambda K}^{p-2} \nabla u^{(\lambda)}_{-1} \rho_1 \right) +  \frac{\lambda^2}{2} |\nabla u^{(\lambda)}_{-1}|_{\lambda \delta, \lambda K}^{p-2} \nabla F_K(u^{(\lambda)}_{-1}) \rho_1 \right) \cdot \pl_\lambda u^{(\lambda)}_{-1} \dd x\,.
\end{gather}
Specializing to $\lambda = 1$, we have 
\begin{equation}
\pl_\lambda  \big|_{\lambda=1} u^{(\lambda)}_{-1}(x)
= \pl_\lambda \big|_{\lambda=1} u_{-\lambda^2}(\lambda x)
= -2 \pl_t u_{-1}(x) + x \nabla u_{-1}(x)\,.
\end{equation}
Dropping the subscript $-1$ to avoid clutter, the previous expression with $\lambda = 1$ can be rewritten as 
\begin{align}
p \int_{\R^n} & \left( -\dv \left( |\nabla u|_{\delta,K}^{p-2} \nabla u \rho_1 \right) + \frac 12 |\nabla u|_{\delta,K}^{p-2} \nabla F_K(u) \rho_1 \right) \cdot (-2 \pl_t u + x \nabla u) \dd x \notag\\
& = p \int_{\R^n} \left( -|\nabla u|_{\delta,K}^{p-2} \nabla u \cdot \nabla \rho_1 
- \dv \left( |\nabla u|_{\delta,K}^{p-2} \nabla u \right) \rho_1 + \frac 12 |\nabla u|_{\delta,K}^{p-2} \nabla F_K(u) \rho_1 \right) \cdot (-2 \pl_t u + x \nabla u) \dd x \notag\\
& = p \int_{\R^n} \left( -|\nabla u|_{\delta,K}^{p-2} \nabla u \cdot \left( -\frac x2 \right) \rho_1 - |\nabla u|_{\delta,K}^{p-2} \pl_t u \rho_1 \right) \cdot (-2 \pl_t u + x \nabla u) \dd x\,,
\end{align}
if we note that $\nabla \rho_1(x) = -\frac{x}{2} \rho_1(x)$, and exploit our equation \eqref{eq:flow-delta-K}. Taking $\rho_1$ and $|\nabla u|_{\delta,K}^{p-2}$ outside the brackets, we note a perfect square above. Together with the contribution from \eqref{eq:additional-contribution}, we have 
\begin{equation}
\pl_\lambda \big|_{\lambda=1} \Phi_{(0,0)}(\lambda^2; u)
= p \int_{\R^n} |\nabla u|_{\delta,K}^{p-2} \left( \frac 12 |2 \pl_t u - x \nabla u|^2 + \delta^2 + F_K(u) \right) \rho_1 \dd x\,.
\end{equation}
This is clearly nonnegative, which shows that $\Phi$ is non-decreasing. 

\medskip

In order to derive a similar formula for general $s$, we write $\mu := s^{1/2}$, take $v := u^{(\mu)}$, and then use the result above: 
\begin{align}
2s \pl_s \Phi_{(0,0)}^{\delta,K}(s; u) 
& = \pl_\lambda \big|_{\lambda=1} \Phi_{(0,0)}^{\delta,K}(\lambda^2 \mu^2; u) \notag\\
& = \pl_\lambda \big|_{\lambda=1} \Phi_{(0,0)}^{\mu \delta, \mu K}(\lambda^2; v) \notag\\
& = p \int_{\R^n} |\nabla v|_{\mu \delta, \mu K}^{p-2} \left( \frac 12 |2 \pl_t v - x \nabla v|^2 + \mu^2 \delta^2 + \mu^2 K^2 F(v) \right) \rho_1 \dd x\,.
\end{align}
To obtain the desired formula for $\pl_s \Phi_{(0,0)}^{\delta,K}(s; u)$, one just needs to express this quantity back in terms of~$u$.  
\end{proof}

\section{$\eps$-regularity for the Ginzburg-Landau approximation} \label{sec:Eps-Reg}

\subsection{The Bochner inequality}

For later applications, it will be important that the energy density $|\nabla u|_{\delta,K}^p$ is a subsolution of a parabolic equation. To obtain such a differential inequality, we require a suitable Bochner formula. 

However, unless $p = 2$, the system \eqref{eq:flow-delta-K} is coupled -- meaning that the coefficients in \eqref{eq:operator-coefficients} are not of the form $a^{k,i}_{\alpha,\beta} = a_{\alpha,\beta}\delta^{i,k}$ -- and this obscures which parabolic operator to apply to $|\nabla u|_{\delta,K}^p$. One can see different choices of operators in the literature, applied to different powers of $|\nabla u|_{\delta,K}$. Let us note for instance the similarity of our choice -- denoted $\cA$ in Theorem \ref{th:Bochner} below -- with $\text{L}_{\text{const.}}$ in \cite[p.~6]{WaLi16}, in the comparable setting of producing Bochner formulas for the weighted $p$-Laplacian on scalar-valued functions and corresponding non-homogeneous heat flow (see also \cite[l.~(2.6)]{WanYanChe13} for the full operator of which $\text{L}$ is the second order piece).

\begin{notation}\label{not:e-u}
To improve readability of the computations, we use the further notational shorthand $e_u := e_{\delta,K}(u)$. This will be used only in this subsection, where $\delta > 0$ and $K < \infty$ are fixed, so there can be no cause for confusion. 
\end{notation}

\begin{theorem}
\label{th:Bochner}
If $u$ is a smooth solution of the $(\delta,K)$-flow \eqref{eq:flow-delta-K}, and $\cA$ is the elliptic operator 
\begin{equation}
\cA \psi := \Delta \psi + (p - 2) |\nabla u|_{\delta,K}^{-2} 
\sum_{i} \nabla^2 \psi(\nabla u^i, \nabla u^i)\,,
\end{equation}
then the $p$-energy density $|\nabla u|_{\delta,K}^p$ satisfies the following Bochner-type formula: 
\begin{equation}
(\pl_t-\cA) |\nabla u|_{\delta,K}^p
\ale_\cN |\nabla u|_{\delta,K}^{p+2}\,.
\end{equation}
In the region where $u$ takes values in the $r_0$-neighborhood of $\cN$, there is an even stronger estimate: 
\begin{equation}
(\pl_t-\cA) |\nabla u|_{\delta,K}^p + c(p) |\nabla u|_{\delta,K}^{p-2} \left( |\nabla^2 u|^2 + K^4 F_u \right)
\ale_\cN |\nabla u|_{\delta,K}^{p+2} 
\end{equation}
with some small constant $c(p) > 0$.
\end{theorem}

This will be an immediate corollary of the explicit formula in Lemma \ref{lem:Bochner-explicit}, which follows by direct computation. We present it in a way that emphasizes \emph{why} we choose to work with the operator $\cA$ applied to the quantity $|\nabla u|_{\delta,K}^{p}$. After this, we explain how one can bound the only non-negative term $(\nabla^2 F_K)_u(\nabla u,\nabla u)$ appearing on the right-hand side. 

\begin{lemma}
\label{lem:Bochner-explicit}
Suppose that $u$ is as in Theorem~\ref{th:Bochner}. Then the following identity holds:
\begin{align}
\label{eq:L-Bochner-formula}
(\partial_t - \cA) e_u^{p/2} 
& = - p e_u^\frac{p-2}{2} (\nabla^2 F_K)_u(\nabla u,\nabla u) - \frac{p}{4} e_u^\frac{p-2}{2} |(\nabla F_K)_u|^2 \notag\\
& \pheq - p e_u^\frac{p-2}{2} |\nabla^2 u|^2 - \frac{p^2(p-2)}{4} e_u^{\frac{p}{2}-3} |\nabla u \cdot \nabla e_u|^2\,.
\end{align}
\end{lemma}

\begin{proof}
Let us start with computing $(\pl-\Delta) e_u$, where $e_u = \delta^2+|\nabla u|^2+F_K(u)$; or even more precisely, with applying the heat operator to terms $|\nabla u|^2$ and $F_K(u)$. In both cases, one proceeds by replacing the term $\Delta u$ with $e_u^{\frac{2-p}{2}} \dv (e_u^{\frac{p-2}{2}} \nabla u) - e_u^{\frac{2-p}{2}} \nabla e_u^{\frac{p-2}{2}} \nabla u$, and leaving the derivatives of $e_u$ unevaluated: 
\begin{align}
(\partial_t - \Delta)|\nabla u|^2 &= 2\ps{\nabla u}{\nabla(\partial_t - \Delta)u} - 2|\nabla^2 u|^2 \notag\\
& = 2\ps{\nabla u}{\nabla\left[(\partial_t - e_u^\frac{2 - p}{2}\dv (e_u^\frac{p - 2}{2}\nabla))u 
+ e_u^{\frac{2-p}{2}} \nabla e_u^{\frac{p-2}{2}} \nabla u \right]} - 2|\nabla^2 u|^2 \notag\\
& = 2\ps{\nabla u}{\nabla\left[-\frac{1}{2}(\nabla F_K)_u + \frac{p - 2}{2}e_u^{-1}\nabla u \cdot \nabla e_u\right]} - 2|\nabla^2 u|^2 \notag\\
& = -(\nabla^2 F_K)_u(\nabla u,\nabla u) -(p - 2)e_u^{-2}|\nabla u\cdot \nabla e_u|^2 \notag\\
& \pheq + (p - 2)e_u^{-1}(\nabla^2u)(\nabla u,\nabla e_u)+(p - 2)e_u^{-1}(\nabla^2 e_u)(\nabla u,\nabla u)- 2|\nabla^2 u|^2\,, \\
(\partial_t - \Delta)F_K(u) 
&= (\nabla F_K)_u \cdot (\partial_t - \Delta)u - (\nabla^2 F_K)_u(\nabla u,\nabla u) \notag\\
&= (\nabla F_K)_u \cdot \left[ (\partial_t - e_u^\frac{2 - p}{2} \dv (e_u^\frac{p - 2}{2} \nabla))u + 
e_u^{\frac{2-p}{2}} \nabla e_u^{\frac{p-2}{2}} \nabla u \right]
- (\nabla^2 F_K)_u(\nabla u,\nabla u) \notag\\
&= -\frac{1}{2}|(\nabla F_K)_u|^2 +\frac{p - 2}{2}e_u^{-1}(\nabla F_K)_u\cdot \nabla u\cdot\nabla e_u - (\nabla^2 F_K)_u(\nabla u,\nabla u)\,.
\end{align}
Now, we combine these two expressions. Recognizing that the terms $(\nabla F_K)_u \nabla u$ and $2 \nabla^2 u \nabla u$ add up to $\nabla e_u$, we have 
\begin{align}
(\pl_t-\Delta) e_u 
& = -\frac 12 |(\nabla F_K)_u|^2 - (p-2)e_u^{-2}|\nabla u \cdot \nabla e_u|^2 - 2|\nabla^2 u|^2 \notag\\
& \pheq - 2 (\nabla^2 F_K)_u(\nabla u,\nabla u) + \frac{p-2}{2} e_u^{-1} |\nabla e_u|^2 + (p-2) e_u^{-1} \nabla^2 e_u (\nabla u,\nabla u)\,.
\end{align}
The first line above collects all nonpositive terms, so now we deal with the remaining terms one by one. First, $(p-2) e_u^{-1} \nabla^2 e_u (\nabla u,\nabla u)$ is the only second-order term (with respect to $e_u$) and so we shift it to the left-hand side, thus introducing the operator $\cA \psi := \Delta \psi + (p - 2) e_u^{-1} \nabla^2 \psi(\nabla u, \nabla u)$ instead of $\Delta$. 
Next, the term $\frac{p-2}{2} e_u^{-1} |\nabla e_u|^2$ is absorbed if we apply the operator $\pl_t-\cA$ to the right power of $e_u$. It is easy to see that one should work with $e_u^{p/2}$, resulting in 
\begin{align}
(\pl_t-\cA) e^{p/2}
= \frac{p}{2} e^{\frac{p-2}{2}} \Bigl( 
& -\frac 12 |(\nabla F_K)_u|^2 - \frac{p(p-2)}{2} e_u^{-2}|\nabla u \cdot \nabla e_u|^2 \notag\\
& - 2|\nabla^2 u|^2 - 2 (\nabla^2 F_K)_u(\nabla u,\nabla u) \Bigr)\,, 
\end{align}
as required. 
\end{proof}

Aiming at the estimates from Theorem \ref{th:Bochner}, we need to investigate the ambiguously signed term $(\nabla^2 F_K)_u(\nabla u,\nabla u)$. This is done using the following lemma. Note that it only gives an upper bound on $-\nabla^2 F_K$ and not on $|\nabla^2 F_K|$, which in general can be much larger. 

\begin{lemma}
\label{lem:Hessian-of-F}
In a tubular neighborhood in which $F(p) = \dist^2(p,\cN)$ is smooth, we have
\begin{equation}
-\nabla^2 F \le C(\cN) \cdot F^{1/2} \cdot I
\end{equation}
in the sense of quadratic forms. In other words, $-\nabla^2 F(p)(v,v) \le C(\cN) \cdot F(p)^{1/2} \cdot |v|^2$ for any $v \in \R^d$. 
\end{lemma}

\begin{proof}
One can rewrite $F$ as $F(p) = |G(p)|^2$, where $G(p) = \pi(p) - p$ is the vector from the point $p$ to its nearest point on $\cN$. A direct calculation shows that
\begin{equation}
\nabla^2 F = 2 G \, \nabla^2 G + \nabla G^T \cdot \nabla G\,.
\end{equation}
Now, $\nabla G^T \cdot \nabla G$ is a positive definite matrix, $|\nabla^2 G|$ is bounded by a constant and $|G| = F^{1/2}$, therefore we arrive at the desired conclusion: 
\begin{equation}
-\nabla^2 F \le 2 |G| \cdot |\nabla^2 G| \cdot I 
\ale_{\cN} F^{1/2} \cdot I\,.
\end{equation}
\end{proof}

\begin{proof}[Proof of Theorem \ref{th:Bochner}]
We begin with the identity from Lemma \ref{lem:Bochner-explicit}, written in terms of $|\nabla u|_{\delta,K}^2 = e_u$, and discard the extra negative terms with no $K^2$-dependence: 
\begin{equation}
\label{eq:reduced-Bochner}
(\partial_t - \cA) |\nabla u|_{\delta,K}^{p}
+ c(p) |\nabla u|_{\delta,K}^{p-2}|\nabla^2 u|^2 
\le p |\nabla u|_{\delta,K}^{p-2} \left(-(\nabla^2 F_K)_u(\nabla u,\nabla u)-\frac{1}{4}|(\nabla F_K)_u|^2\right)\,.
\end{equation}
Depending on how far $u(t,x)$ is from $\cN$, there are three possibilities, the last one being the most important: 
\begin{enumerate}

\item If $u$ does not lie in the $2r_0$-neighborhood of $\cN$, then $F_K$ and its derivatives vanish, and \eqref{eq:reduced-Bochner} reduces to $(\partial_t - \cA) |\nabla u|_{\delta,K}^{p} \le 0$. 

\item If $u$ lies in the intermediate region $r_0 < \dist(u,\cN) < 2r_0$, then we bound the derivatives of $\dist(\cdot,\cN)$ by constants. Since $F$ has some definite size (between $r_0^2$ and $(2r_0)^2$), one can apply a~naive bound 
\begin{equation}
- \nabla^2 F_K (\nabla u,\nabla u) - \frac{1}{4}|\nabla F_K|^2 
\ale_\cN K^2 |\nabla u|^2 
\ale_{r_0} (\delta^2 + K^2 F_u + |\nabla u|^2)^2 
= |\nabla u|_{\delta,K}^4 \,.
\end{equation}
Together with \eqref{eq:reduced-Bochner}, this leads to $(\pl_t - \cA) |\nabla u|_{\delta,K}^{p} \ale_{\cN} |\nabla u|_{\delta,K}^{p+2}$.

\item If $u$ lies in the $r_0$-neighborhood of $\cN$, then $F_K(u) = K^2 \dist^2(u,\cN)$ and we can apply Lemma \ref{lem:Hessian-of-F} together with $|\nabla F_K| = 2 K^2 F^{1/2}$: 
\begin{equation}
-2(\nabla^2 F_K)_u(\nabla u,\nabla u)-\frac{1}{2}|(\nabla F_K)_u|^2 
\le C_1 K^2 F_u^{1/2} |\nabla u|^2 - 2 K^4 F_u 
\le C_2 |\nabla u|^4 - K^2 (F_K)_u\,.
\end{equation}
Finally,  
\begin{equation}
(\pl_t-\cA) |\nabla u|_{\delta,K}^p + c(p) |\nabla u|_{\delta,K}^{p-2} \left( |\nabla^2 u|^2 + K^2 (F_K)_u \right)
\ale_\cN |\nabla u|_{\delta,K}^{p+2}\,.
\end{equation}

\end{enumerate}
\end{proof}

\subsection{Parabolic Harnack inequality}

In this subsection, we recall a version of parabolic Harnack inequality proved by Ferretti and Safonov \cite{FerSaf01}, that is suitable for non-divergence type operators. 

\begin{theorem}
\label{th:Parabolic-Harnack}
As in the Bochner formula (Theorem \ref{th:Bochner}), let $\cA$ be the operator 
\begin{equation}
\cA \psi := \Delta \psi + (p - 2) |\nabla u|_{\delta,K}^{-2} \nabla^2 \psi(\nabla u, \nabla u)\,,
\end{equation}
where $u$ is a $C^{2,1}$-regular map. If $v$ is a nonnegative $C^{2,1}$ function satisfying 
\begin{equation}
(\pl_t - \cA) v \le 0 
\end{equation}
in some parabolic cylinder $P_r(z)$, then 
\begin{equation}
v(z) \le \frac{C(n,p)}{|P_r(z)|} \int_{P_r(z)} v \dd x \dd t\,. 
\end{equation}
\end{theorem}

Indeed, this is a special case of \cite[Thm.~3.4]{FerSaf01}, obtained by specializing to the case $v \ge 0$, choosing the $L^1$-norm for the final estimate and verifying that our operator $\cA$ satisfies the ellipticity condition. In later applications, the parabolic Harnack inequality will always be used together with the Bochner formula. Hence, the following combination of these two theorems will be useful: 

\begin{corollary}
\label{cor:Bochner-Harnack-combined}
If $u$ is a smooth solution of the $(\delta,K)$-flow \eqref{eq:flow-delta-K} and $|\nabla u|_{\delta,K} \le A$ in a~parabolic cylinder $P_r(z)$, then 
\begin{equation}
|\nabla u|_{\delta,K}^p(z) \le \frac{C e^{A^2 C r^2}}{|P_r(z)|} \int_{P_r(z)} |\nabla u|_{\delta,K}^p \dd x \dd t
\end{equation}
with $C(n,\cN,p) > 0$. 
\end{corollary}

For future use, we will always use a constant $A(n,\cN,p) > 0$ and $r = 1$, so the constant on the right hand side is again just $C(n,\cN,p) > 0$. 

\begin{proof}
If we denote $w := |\nabla u|_{\delta,K}^p$, the Bochner formula (Theorem \ref{th:Bochner}) gives us 
\begin{equation}
(\pl_t-\cA) w
\le C |\nabla u|_{\delta,K}^{p+2} \le A^2 C w 
\quad \text{in } P_r(z)\,,
\end{equation}
which almost meets the assumptions of Theorem \ref{th:Parabolic-Harnack}. Thus, consider the auxiliary function 
\begin{equation}
v(t,x) := e^{-A^2 C (t-t_0)} w(t,x)\,,
\end{equation}
where we used $t_0$ to denote the time coordinate of $z$. Now, $v$ is a nonnegative regular function satisfying $(\pl_t-\cA) v \le 0$, so by Theorem \ref{th:Parabolic-Harnack} we have 
\begin{equation}
v(z) 
\le \frac{C}{|P_r(z)|} \int_{P_r(z)} v \dd x \dd t\,.
\end{equation}
Taking into account $w(z) = v(z)$ and $v(t,x) \le e^{A^2 C r^2} w(t,x)$ in $P_r(z)$, we infer a similar inequality for $w$: 
\begin{equation}
w(z) \le \frac{C e^{A^2 C r^2}}{|P_r(z)|} \int_{P_r(z)} w \dd x \dd t\,.
\end{equation}
\end{proof}

\subsection{The $\eps$-regularity theorem and related results}.

We now turn to $\eps$-regularity for the $(\delta,K)$-flow \eqref{eq:flow-delta-K}. 
For nonlinear equations arising in geometry, it is common that solutions can develop singularities, and thus one is forced to look for conditional regularity results that assume the energy is $\eps$-small. In the context of the $p = 2$ harmonic map flow such an $\eps$-regularity theorem was proved by Struwe \cite[Th.~5.1]{Str88} and Theorem \ref{th:eps-regularity} below is its strict analogue.

\begin{theorem}[$\eps$-regularity]
\label{th:eps-regularity}
There is a constant $\eps_0(n,\cN,p) > 0$ such that if $u \colon [-1,0] \times \R^n \to \R^d$ is a~smooth solution of the $(\delta,K)$-flow \eqref{eq:flow-delta-K} and $\Phi^{\delta,K}_{(0,0)}(1) \le \eps_0$, then 
\begin{equation}
|\nabla u|_{\delta,K} \le C 
\quad \text{in } P_\gamma
\end{equation}
for some positive constants $C,\gamma$ depending only on $n$, $p$, $\cN$ and $E_0 := \Phi_{(1,0)}^{\delta,K}(2) \approx_p H_2 |\nabla u_{-1}|^p_{\delta,K}(0)$. 
\end{theorem}

\begin{proof}
Fix $\gamma > 0$ to be chosen later. The function 
\begin{equation}
[\gamma/2,\gamma] \ni \rho \longmapsto (\gamma-\rho) \sup_{P_\rho} |\nabla u|_{\delta,K}
\end{equation}
is continuous and hence achieves its maximum at some point $\rho \in [\gamma/2,\gamma)$ (since $\rho = \gamma$ would imply that $|\nabla u|_{\delta,K} \equiv 0$ in $P_\gamma$). Let us pick such $\rho$, as well as $z_0 = (t_0,x_0) \in \overline{P_\rho}$ for which this maximum is achieved: 
\begin{equation}
|\nabla u|_{\delta,K}(z_0) = \sup_{P_\rho} |\nabla u|_{\delta,K} =: c_0\,. 
\end{equation}
There are two possibilities: 

\textsc{Case 1.} If $c_0 \le \frac{2}{\gamma-\rho}$, then 
\begin{equation}
(\gamma/2) \sup_{P_{\gamma/2}} |\nabla u|_{\delta,K}
\le (\gamma-\rho) \sup_{P_\rho} |\nabla u|_{\delta,K}
\le 2\,,
\end{equation}
which leads to the desired bound on $P_{\gamma/2}$, if only $\gamma$ depends only on $n,\cN,E_0$ (which we will see later). 

\textsc{Case 2.} If $c_0 \ge \frac{2}{\gamma-\rho}$, we will establish a contradiction for small $\eps_0$ and $\gamma$. To this end, start by noting that by maximality of $\rho$, 
\begin{equation}
(\gamma-\rho) \sup_{P_\rho} |\nabla u|_{\delta,K}
\ge \left( \gamma-\frac{\gamma+\rho}{2} \right) \sup_{P_{\frac{\gamma+\rho}{2}}} |\nabla u|_{\delta,K} 
= \frac{\gamma-\rho}{2} \sup_{P_{\frac{\gamma+\rho}{2}}} |\nabla u|_{\delta,K}\,.
\end{equation}
By cancelling $\gamma-\rho$, it follows that $|\nabla u|_{\delta,K} \le 2c_0$ on $P_{\frac{\gamma+\rho}{2}}$. 

Consider now the rescaled map $v(t,x) := u(t_0+\frac{t}{c_0^2}, x_0+\frac{x}{c_0})$, which solves the $(\delta',K')$-approximating equation with $\delta'=\delta/c_0$, $K'=K/c_0$. By definition, we have $|\nabla v|_{\delta',K'}(0,0) = 1$. Moreover, 
\begin{equation}
\sup_{P_1} |\nabla v|_{\delta',K'}
= c_0^{-1} \sup_{P_{c_0^{-1}}(z_0)} |\nabla u|_{\delta,K}
\le c_0^{-1} \sup_{P_{\frac{\gamma+\rho}{2}}} |\nabla u|_{\delta,K}
\le 2\,.
\end{equation}
The inclusion $P_{c_0^{-1}}(z_0) \subseteq P_{\frac{\gamma+\rho}{2}}$ above is guaranteed by the assumption $c_0 \ge \frac{2}{\gamma-\rho}$. 

In contrast, the assumption $\Phi^{\delta,K}_{(0,0)}(1) \le \eps_0$ together with the monotonicity formula implies that $|\nabla v|_{\delta',K'}$ is $L^p$-small in $P_1$. 
To see this, start by comparing the $L^p$-norm of $|\nabla v_t|_{\delta',K'}^p$ with the monotone quantity $\Phi_{(1,0)}^{\delta',K'}(1-t;v)$. Since $-1 \le t \le 0$ in $P_1$, the weights $1$ and $(1-t)^{p/2} \rho_{1-t}$ are comparable and hence 
\begin{equation}
\int_{P_1} |\nabla v_t|_{\delta',K'}^p \dd t \dd x 
\ale_{n,p} \int_{P_1} |\nabla v_t|_{\delta',K'}^p (1-t)^{p/2} \rho_{1-t} \dd t \dd x 
\le \int_{-1}^0 \Phi_{(1,0)}^{\delta',K'}(1-t;v) \dd t\,.
\end{equation}
Exploiting monotonicity and rescaling back, we then have 
\begin{equation}
\int_{P_1} |\nabla v_t|_{\delta',K'}^p \dd t \dd x 
\le \Phi_{(1,0)}^{\delta',K'}(2; v) 
= \Phi_{(t_0+c_0^{-2},x_0)}^{\delta,K}(2c_0^{-2}; u) 
\le \Phi_{(t_0+c_0^{-2},x_0)}^{\delta,K}(1+t_0+c_0^{-2}; u) \,.
\end{equation}
Now denote the point $z_1 = (t_1,x_1) := (t_0+c_0^{-2},x_0)$ and notice that it is close to $(0,0)$: 
\begin{equation}
|x_1| = |x_0| \le \gamma\,, \qquad 
|t_1| \le \max(-t_0,c_0^{-2}) \le \gamma^2\,.
\end{equation}
We are left with comparing the two following quantities: 
\begin{align}
\Phi_{z_1}^{\delta,K}(1+t_1; u) 
& = (1+t_1)^{p/2} \int_{\R^n} |\nabla u_{-1}|_{\delta,K}^p(x) \rho_{1+t_1}(x-x_0) \dd x \\
\text{and} \qquad 
\Phi_{(0,0)}^{\delta,K}(1; u)
& = \int_{\R^n} |\nabla u_{-1}|_{\delta,K}^p(x) \rho_{1}(x) \dd x\,.
\end{align}
The integrands are $\gamma$-close at unit scale, but have incomparable tails, so to control their difference, we need to use another Gaussian function with a heavier tail: 
\begin{equation}
\Phi_{z_1}^{\delta,K}(1+t_1) 
\ale_p \Phi_{(0,0)}^{\delta,K}(1)
+ \int_{\R^n} |\nabla u_{-1}|_{\delta,K}^p(x) |\rho_{1+t_1}(x-x_0) - \rho_{1}(x)| \dd x 
\ale_n \eps_0 + \gamma \Phi_{(1,0)}^{\delta,K}(2)\,.
\end{equation}

\medskip

Let us summarize the properties of $v$: 
\begin{itemize}
\item $v$ is a smooth solution of the $(\delta',K')$-approximating equation \eqref{eq:flow-delta-K};
\item $\displaystyle \int_{P_1} |\nabla v_t|_{\delta',K'}^p \dd t \dd x \ale_{n,p} \eps_0 + \gamma \cdot E_0$ (which can be taken arbitrarily small);
\item $\displaystyle |\nabla v|_{\delta',K'} \le 2$ in $P_1$;
\item $\displaystyle |\nabla v|_{\delta',K'} = 1$ at the point $(0,0)$. 
\end{itemize}
The contradiction now follows from a combination of the Bochner formula and the parabolic Harnack inequality (Corollary \ref{cor:Bochner-Harnack-combined}). Since $|\nabla v|_{\delta',K'} \le 2$, it gives us 
\begin{equation}
|\nabla v|_{\delta',K'}^p(0,0) \ale_{n,\cN,p} \int_{P_1} |\nabla v|_{\delta',K'}^p \dd t \dd x 
\ale_{n,p} \eps_0 + \gamma \cdot E_0\,.
\end{equation}
Taking appropriate small $\eps_0(n,\cN,p) > 0$ and $\gamma(n, \cN, p, E_0) > 0$, we infer that $|\nabla v|_{\delta',K'}^p(0,0) < 1$, yielding a~contradiction with $|\nabla v|_{\delta',K'}(0,0) = 1$. 
\end{proof}

This result can be paired with the energy inequality to obtain $K$-independent $C^{1,1/2}$ estimates, which will later be necessary in obtaining solutions of the $\delta$-flow \eqref{eq:reg-flow} in the limit $K \to \infty$. In brief, space-wise Lipschitz bounds $|\nabla u| \le M$ imply spacetime Lipschitz bounds (with respect to the parabolic metric), as detailed in the following lemma. 

\begin{lemma}
\label{lem:grad-to-Lip}
If for $\gamma \in (0,1]$ and a solution $u: P_\gamma \to \R^d$ to the $(\delta,K)$-flow \eqref{eq:flow-delta-K} we assume that
\begin{equation}
\sup_{P_\gamma}|\nabla u|_{\delta,K} \leq M\,,
\end{equation}
then one has $[u]_{\Lip(P_{\gamma/4})} \leq C(n, M, p, \delta)$. More specifically, the parabolic Lipschitz seminorm is controlled by
\begin{equation}
|u(t_1,x_1) - u(t_2,x_2)| \ale_{n, p} M|x_1 - x_2| + (M + \delta^{\frac{2 - p}{2}}M^{p/2})|t_1 - t_2|^{1/2}\qquad \text{for }(t_1,x_1),(t_2,x_2) \in P_{\gamma/4}\,.
\end{equation}
\end{lemma}

\begin{proof}
We shall compare $u(t_1,x_1)$, $u(t_2,x_2)$ by comparing the averages $(u_{t_1})_{B_r(x_0)}$, $(u_{t_2})_{B_r(x_0)}$, where we choose 
\begin{equation}
r := |t_1-t_2|^{1/2}\,, \qquad
x_0 := \frac{x_1+x_2}{2}\,, \qquad 
t_0 := \max(t_1,t_2)\,.
\end{equation}
By triangle inequality, 
\begin{equation}
|u(t_1,x_1) - u(t_2,x_2)| 
\le |u(t_1,x_1)-(u_{t_1})_{B_r(x_0)}| + |u(t_2,x_2)-(u_{t_2})_{B_r(x_0)}|
+ |(u_{t_1})_{B_r(x_0)}-(u_{t_2})_{B_r(x_0)}|\,.
\end{equation}
The first two terms are bounded directly using the assumption $|\nabla u| \le M$ (even if $x_1,x_2$ lie outside $B_r(x_0)$). The last term can be bounded by 
\begin{align}
|(u_{t_1})_{B_r(x_0)}-(u_{t_2})_{B_r(x_0)}| 
& \le \fint_{B_r(x_0)} |u_{t_1}-u_{t_2}| \dd x \notag\\
& \le \fint_{B_r(x_0)} \int_{t_0-r^2}^{t_0} |\pl_t u| \dd t \dd x \notag\\
& \ale r \left( r^{-n} \int_{P_r(t_0,x_0)} |\pl_t u|^2 \dd t \dd x \right)^{1/2}\,,
\end{align}
so that it is enough to estimate the last integral in terms of $M$. This is possible if we apply the naive bound $|\nabla u|_{\delta,K} \ge \delta$; a rescaled version of Lemma \ref{lem:loc-energy-ineq} then implies: 
\begin{equation}
r^{-n} \int_{P_r(t_0,x_0)} |\pl_t u|^2 
\le \delta^{2-p} r^{-n} \int_{P_r(t_0,x_0)} |\nabla u|^{p-2}_{\delta,K} |\pl_t u|^2 
\ale_{p,n} \delta^{2-p} r^{-n-2} \int_{P_{2r}(t_0,x_0)} |\nabla u|^{p}_{\delta,K}
\ale_n \delta^{2-p} M^p\,.
\end{equation}
Summing up the contributions of all terms, we arrive at the desired estimate.
\end{proof}

In the context of $\eps$-regularity regions, this implies that

\begin{corollary}
\label{cor:eps-reg-Lip}
Under the hypotheses of Theorem \ref{th:eps-regularity}, one has $[u]_{\Lip(P_{\gamma/4})} \leq C(n,\cN, E_0, \delta)$. More specifically, the parabolic Lipschitz seminorm is controlled by
\begin{equation}
|u(t,x) - u(s,y)| \ale_{n,\cN,E_0} |x - y| + \delta^{\frac{2 - p}{2}}|t - s|^{1/2}\qquad \text{for }(t,x),(s,y) \in P_{\gamma/4}\,.
\end{equation}
\end{corollary}

\begin{remark}
The $[\cdot]_\Lip$ seminorm estimate in the conclusion of Corollary \ref{cor:eps-reg-Lip} degenerates as $\delta \to 0$. A~natural improvement of this lemma would be to remove the $\delta$-dependence of the $[\cdot]_\Lip$ estimate. Let us note however that: 1) this is beyond the scope of this paper, where $\delta > 0$ is fixed throughout, and 2) methods involving the maximum principle, such as are used in \cite[Lem.~4.3]{JinSil17} to prove an analogous result in the scalar-valued ($\R = \cN = \R^d$) case, would require a new idea given that the $(\delta,K)$-flow is a system of $d$ equations coupled to highest order.
\end{remark}

The $\eps$-regularity theorem above shows that $u$ is regular on a parabolic cylinder centered at $(0,0)$, but it does not imply directly that it is regular on a neighborhood of $(0,0)$. The following lemma fills this gap by showing that a pointwise bound $|\nabla u|_{\delta,K} \le C$ is preserved for a small but uniformly positive time. 

\begin{lemma}
\label{lem:eps-reg-for-later-time}
There is a constant $\gamma(n,\cN,p) > 0$ such that if $u$ is a regular solution of the $(\delta,K)$-flow \eqref{eq:flow-delta-K} and $|\nabla u_0|_{\delta,K} \le \gamma$ on the ball $B_2$, then 
\begin{equation}
|\nabla u|_{\delta,K} \le 4 
\quad \text{on } [0,r_0] \times B_1
\end{equation}
for some $r_0(n,\cN,p,E_0) > 0$, where $E_0 := H_1 |\nabla u|^p_{\delta,K}(0)$. 
\end{lemma}

A common assumption in the applications is of the form $|\nabla u_0|_{\delta,K} \le B$ (where $B$ is not necessarily small). Rescaling (by a factor of $B/\gamma$), one can apply the lemma above and (after rescaling again) conclude that $|\nabla u|_{\delta,K} \le C(n,\cN) B$ on a small region. The size of that region depends on $B$, but this is not an issue in practice. 

\begin{proof}
We shall adapt the proof of Theorem \ref{th:eps-regularity}. Fix $r_0 > 0$ to be chosen later, depending on $n$, $\cN$ and $E_0$. Consider a point $q := (r^2,x)$ in the required region, with $0 < r \le r_0$ and $x \in B_1$. As before, we choose $\rho \in [r/2,r]$ maximizing the function 
\begin{equation}
[r/2,r] \ni \rho \longmapsto (r-\rho) \sup_{P_\rho(q)} |\nabla u|_{\delta,K}\,,
\end{equation}
and then a point $z_0 = (t_0,x_0) \in \ov{P_\rho(q)}$ for which this supremum is achieved: 
\begin{equation}
|\nabla u|_{\delta,K}(z_0) = \sup_{P_\rho(q)} |\nabla u|_{\delta,K} =: c_0 \cdot r\,. 
\end{equation}
The rest of the proof follows the reasoning from Theorem \ref{th:eps-regularity}, but we rescale by the factor of $c_0$ rather than $c_0 \cdot r$ (which is the bound on $|\nabla u|_{\delta,K}$ in $P_\rho(q)$). 

\medskip

If $c_0 \le \frac{2}{r-\rho}$, then 
\begin{equation}
(r/2) \sup_{P_{r/2}(q)} |\nabla u|_{\delta,K} 
\le (r-\rho) \sup_{P_\rho(q)} |\nabla u|_{\delta,K} 
\le 2r\,,
\end{equation}
and in particular $|\nabla u|_{\delta,K}(q) \le 4$, as required.

\medskip

The remaining case leads to a contradiction. As before, we consider the rescaled map $v(t,x) = u(t_0+\frac{t}{c_0^2}, x_0+\frac{x}{c_0})$ (which solves the $(\delta',K')$-approximating equation with $\delta' = \delta/c_0$, $K' = K/c_0$) and use Corollary \ref{cor:Bochner-Harnack-combined} to establish the chain of inequalities 
\begin{equation}
r^p = |\nabla v|_{\delta',K'}^p(0,0) 
\ale_{n,\cN,p} \int_{P_1} |\nabla v|_{\delta',K'}^p \dd t \dd x 
\ale_{n,p} \int_{-1}^0 \Phi_{(1,0)}^{\delta',K'}(1-t;v) \dd t\,. 
\end{equation}
Then, monotonicity formula implies that 
\begin{equation}
r^p 
\ale_{n,\cN,p} \Phi_{(1,0)}^{\delta',K'}(2; v) 
= \Phi_{z_1}^{\delta,K}(2c_0^{-2}; u) 
\le \Phi_{z_1}^{\delta,K}(t_1; u) 
= t_1^{p/2} \cdot H_{t_1} |\nabla u_0|_{\delta,K}^p(x_0)\,.
\end{equation}
where again $z_1 = (t_1,x_0) = z_0 + (c_0^{-2},0)$. Note that $t_1 \le 2r^2$, so the above can be summarized as 
\begin{equation}
r^p \le B(n,\cN,p) \cdot r^p H_{t_1} |\nabla u_0|_{\delta,K}^p(x_0)\,.
\end{equation}
To arrive at a contradiction, it is enough to take $\gamma^p := \frac{1}{4B}$ and choose $r_0 > 0$ small enough, so that our assumption $|\nabla u_0|_{\delta,K} \le \gamma$ implies $H_{t_1} |\nabla u_0|_{\delta,K}^p(x_0) \le 2 \gamma^p$. Therefore, the proof is completed by an application of the simple lemma below. 
\end{proof}

\begin{lemma}
Assume that the map $u_0$ satisfies $|\nabla u_0|_{\delta,K} \le \gamma$ in the ball $B_2$, and let $E_0 := H_1 |\nabla u_0|^p_{\delta,K}(0)$. Then there is a radius $r_0(n,E_0,\gamma,p) > 0$ for which 
\begin{equation}
H_{r^2} |\nabla u_0|_{\delta,K}^p(x) \le 2 \gamma^p
\quad \text{for } x \in B_1, \ 0 < r \le r_0\,.
\end{equation}
\end{lemma}

\begin{proof}
Intuitively, $H_{r^2} |\nabla u_0|_{\delta,K}^p(x)$ is comparable to the average $\fint_{B_r(x)} |\nabla u_0|_{\delta,K}^p$, which is at most $\gamma^p$. Thus, one only needs to handle the Gaussian tail: 
\begin{align}
H_{r^2} |\nabla u_0|_{\delta,K}^p(x)
& = \int_{B_2} |\nabla u(y)|_{\delta,K}^p \rho_{r^2}(x-y) \dd y 
+ \int_{B_2^c} |\nabla u(y)|_{\delta,K}^p \rho_{r^2}(x-y) \dd y \notag\\
& \le \gamma^p \int_{B_2} \rho_{r^2}(x-y) \dd y 
+ \sup_{x \in B_1, \, y \in B_2^c} \frac{\rho_{r^2}(x-y)}{\rho_1(y)} 
\cdot \int_{B_2^c} |\nabla u(y)|_{\delta,K}^p \rho_1(y) \dd y \notag\\
& \le \gamma^p +  r^{-n} e^{-\frac{1}{4r^2}+1} \cdot E_0\,.
\end{align}
The last inequality easily follows once we note that $|x-y| \ge \frac 12 |y|$. In result, the $r$-dependent term tends to zero as $r \to 0$, and one can assume $r \le r_0(n,\gamma,E_0,p)$ so that both terms together give at most $2\gamma^p$. 
\end{proof}

\subsection{Hessian bounds in $\eps$-regularity regions}

In this subsection, we obtain bounds on the signed second order terms appearing in the Bochner formula, assuming the hypotheses and conclusion of Theorem \ref{th:eps-regularity}. For $p = 2$, the divergence form of the second-order part of the operator makes this a simple matter of integration by parts, but that is not the case for $p > 2$. However, the following lemmas show that an integration by parts can be performed up to some error that is not necessarily of lower order, but that can be absorbed by other (signed) terms of the Bochner formula. The end result can be stated as follows, using the shorthand Notation \ref{not:e-u}, which we will prove after having built up the necessary tools via a sequence of lemmas:

\begin{theorem}
\label{th:eps-reg-W22}
Let $u$ be a smooth solution of \eqref{eq:flow-delta-K} in $P_2$, and assume that $e_u(x) \leq M^2$ in $P_2$. We then have access to weighted $W^{2,2}$ bounds
\begin{equation}
\label{eq:eps-reg-W22-eq}
\int_{P_{1/2}} (|(\nabla F_K)_u|^2 + |\nabla^2 u|^2)e_u^{\frac{p-2}{2}} \leq C(p,n)\left[\int_{P_2} e_u^{p/2} + \int_{P_2} e_u^\frac{p + 2}{2}\right] \leq C(p,n,M)\,.
\end{equation}
\end{theorem}

\begin{remark}
In particular, the assumption implies that $K^2 F(u) \le M^2$. Thus for large $K$ it follows that $F(u) = \dist^2(u,\cN)$, and that $u$ takes values in a very small $(M/K)$-neighborhood of $\cN$. 
\end{remark}

\begin{remark}
A similar estimate is exploited by Chen and Struwe \cite{CheStr89} in the case $p = 2$. There, it is enough to test the Bochner formula 
\begin{equation}
(\pl_t-\Delta) |\nabla u|_{K}^2 + c \left( K^4 F(u) + |\nabla^2 u|^2 \right)
\ale |\nabla u|_{K}^{4}
\end{equation}
with a cut-off function $\chi^2$. Since $|\nabla u|_K$ is assumed to be bounded, integration by parts lets us put all other derivatives on $\chi^2$ and in result the integral of $K^4 F(u) + |\nabla^2 u|^2$ is bounded. In our case, there is an additional technical difficulty. Our Bochner formula involves a non-divergence form operator $\cA$, so integration by parts is not possible directly. However, we can transform the non-divergence part of $\cA$ into a lower order term. In the following, we do exactly that. 
\end{remark}

\begin{lemma}
\label{lem:stationary-eq}
Let $u$ be a smooth solution of \eqref{eq:flow-delta-K} in $P_1$, and $X \in C_c^{2,1}((-1,0]\times B_1^n, \R^n)$ a vector field. One then has the stationary-type equation
\begin{equation}
-\int (\partial_t u) \cdot du(X) e_u^\frac{p - 2}{2} 
= \int \ps{du^T \cdot du}{\nabla X}e_u^\frac{p - 2}{2} 
- \frac 12 \int e_u \cdot \dv \left( e_u^{\frac{p-2}{2}} X \right)\,.
\end{equation}
\end{lemma}

\begin{proof}
The idea, as in the elliptic setting, is to pair \eqref{eq:flow-delta-K} with $du(X) e_u^{\frac{p-2}{2}}$ and integrate by parts in space. Since $X$ is compactly supported, no boundary terms appear. Indeed, doing so gives
\begin{align}
\int (\partial_t u) \cdot du(X) e_u^\frac{p - 2}{2} &= \int (\dv(e_u^\frac{p - 2}{2} \nabla u)\cdot du(X) - \frac{1}{2}\int d(F_K(u))(X) e_u^\frac{p - 2}{2}\notag\\
&= -\int \ps{\nabla u}{\nabla(du(X))}e_u^\frac{p - 2}{2} + \frac{1}{2}\int F_K(u)\dv(e_u^\frac{p - 2}{2}X) \notag\\
&=-\int (\nabla^2 u)(du, X)e_u^\frac{p - 2}{2} -\int \ps{du^T \cdot du}{\nabla X}e_u^\frac{p - 2}{2} \notag\\
& \pheq + \frac{1}{2}\int F_K(u)\dv(e_u^\frac{p - 2}{2}X)\,.
\end{align}
Let us now focus on the first term on the r.h.s, which we will integrate by parts
\begin{equation}
-\int (\nabla^2 u)(du, X)e_u^\frac{p - 2}{2} = -\int X\cdot \nabla \left( \frac{1}{2}(|du|^2 + \delta^2) \right) e_u^\frac{p - 2}{2} 
= \frac{1}{2}\int (|du|^2 + \delta^2)\dv(e_u^\frac{p - 2}{2}X) \,.
\end{equation}
Now, replacing $-\int (\nabla^2 u)(du, X)e_u^\frac{p - 2}{2}$ in the previous r.h.s with this final expression, and grouping $F_K(u)$ and $|du|^2 + \delta^2$ terms into $e_u$, we obtain the desired conclusion.
\end{proof}

\begin{lemma}
\label{lem:L-IBP}
Let $u$ be a smooth solution of \eqref{eq:flow-delta-K} in $P_1$, and $\chi \in C_c^{2,1}((-1,0]\times B_1^n)$ a cutoff. One then has the identity up to lower order error
\begin{align}
-\frac{p}{2}\int e_u^{-1}(\partial_t u) \cdot du(\nabla e_u) e_u^\frac{p - 2}{2}\chi^2 &= \int e_u^{-1}\ps{du^T \cdot du}{\nabla^2 e_u^{p/2}}\chi^2 \notag\\
& \pheq -\frac{p^2}{4}\int e_u^{-1}|du(\nabla e_u)|^2e_u^\frac{p - 4}{2}\chi^2 \notag\\
& \pheq + \frac{p}{4}\int e_u^{-1}|\nabla e_u|^2e_u^\frac{p - 2}{2}\chi^2 + E\,,
\end{align}
where $|E| \leq p \int |\nabla e_u|e_u^{\frac{p - 2}{2}}|\nabla \chi|\chi$.
\end{lemma}

\begin{proof}
Applying Lemma \ref{lem:stationary-eq} with $X := \frac{p}{2}\chi^2 e_u^{-1}\nabla e_u = \chi^2 e_u^{-p/2}\nabla e_u^{p/2}$ yields
\begin{align}
-\frac{p}{2}\int e_u^{-1}(\partial_t u) \cdot du(\nabla e_u) e_u^\frac{p - 2}{2}\chi^2 &= \int \ps{du^T \cdot du}{\nabla (\chi^2 e_u^{-p/2}\nabla e_u^{p/2})}e_u^\frac{p - 2}{2} \notag\\
& \pheq - \frac 12 \int e_u \cdot \dv \left( \chi^2 e_u^{-1} \nabla e^{p/2} \right) \notag\\
& =: (I) + (II)\,.
\end{align}
Now, examining $(I)$ gives
\begin{align}
(I) &=\int e_u^{-1}\ps{du^T \cdot du}{\nabla^2 e_u^{p/2}}\chi^2 -\frac{p}{2}\int e_u^{-2} \ps{du^T \cdot du}{\nabla e_u\otimes \nabla e_u^{p/2})} \chi^2 \notag\\
& \pheq +2\int e_u^{-1}\ps{du^T \cdot du}{\nabla\chi\otimes \nabla e_u^{p/2}}\chi \notag\\
&=\int e_u^{-1}\ps{du^T \cdot du}{\nabla^2 e_u^{p/2}}\chi^2 -\frac{p^2}{4}\int e_u^{-1}|du(\nabla e_u)|^2e_u^\frac{p - 4}{2}\chi^2 \notag\\
& \pheq +p\int e_u^{-1}\ps{du^T \cdot du}{\nabla\chi\otimes \nabla e_u}e_u^\frac{p - 2}{2}\chi\,.
\end{align}

Note that the third term on the r.h.s takes the form required of $E$. As for $(II)$, 
\begin{align}
(II) 
& = - \frac 12 \int \dv \left( \chi^2 \nabla e^{p/2} \right) 
- \frac 12 \int e_u \cdot \chi^2 \nabla e^{p/2} \nabla e_u^{-1} \notag\\
&= \frac{p}{4}\int e_u^{-1}|\nabla e_u|^2e_u^\frac{p - 2}{2}\chi^2 \,,
\end{align}
where we exploited the spatial compactness of $\supp\chi$ to write the integral of a divergence as $0$ above. 
\end{proof}

\begin{proof}[Proof of Theorem \ref{th:eps-reg-W22}]
Let $\chi \in C_c^\infty((-\infty,0] \times \R^n)$ be a smooth cutoff satisfying $\mathbbm{1}_{P_{1/2}} \le \chi \le \mathbbm{1}_{P_1}$ with derivatives $\nabla \chi$, $\pl_t \chi$ bounded by dimensional constants.

We multiply the Bochner formula \eqref{eq:L-Bochner-formula} by $\chi^2$ and integrate over $P_1$. Recalling the precise form of the operator $\cA$, this gives us
\begin{align}
&\int (\partial_t e_u^{p/2})\chi^2 - \int (\Delta e_u^{p/2})\chi^2 - (p - 2)\int e_u^{-1}\ps{du^T\cdot du}{\nabla^2 e_u^{p/2}}\chi^2 \notag\\
&= -p\int e_u^{\frac{p - 2}{2}}(\nabla^2 F_K)_u(\nabla u,\nabla u)\chi^2 - \frac{p}{4}\int e_u^{\frac{p - 2}{2}} |(\nabla F_K)_u|^2\chi^2 - p\int e_u^{\frac{p - 2}{2}}|\nabla^2 u|^2\chi^2 \notag\\
& \pheq - \frac{p^2(p - 2)}{4}\int e_u^{\frac{p}{2} - 3}|du(\nabla e_u)|^2\chi^2\,.
\end{align}
As already mentioned, there is one problematic term that does not appear in the case $p = 2$: the third term on the l.h.s. Together with the final term on the r.h.s. it can now be rewritten according to the above Lemma~\ref{lem:L-IBP}:
\begin{align}
& (p - 2)\int e_u^{-1}\ps{du^T\cdot du}{\nabla^2 e_u^{p/2}}\chi^2 - \frac{p^2(p - 2)}{4}\int e_u^{\frac{p}{2} - 3}|du(\nabla e_u)|^2\chi^2 \notag\\
&= -(p - 2)\left(\frac{p}{2}\int e_u^{-1}(\partial_t u) \cdot du(\nabla e_u) e_u^\frac{p - 2}{2}\chi^2 + \frac{p}{4}\int e_u^{-1}|\nabla e_u|^2e_u^\frac{p - 2}{2}\chi^2 + E\right)\,.
\end{align}
In summary, we were able to rewrite the terms of interest as
\begin{align}
\label{eq:Hess-eq}
\int e_u^{\frac{p - 2}{2}} \left( \frac p4 |(\nabla F_K)_u|^2 + p |\nabla^2 u|^2 \right) \chi^2 
&= -p\int e_u^{\frac{p - 2}{2}}(\nabla^2 F_K)_u(\nabla u,\nabla u)\chi^2 - \int(\partial_t e_u^{p/2})\chi^2 \notag\\
& \pheq + \int (\Delta e_u^{p/2})\chi^2 -\frac{p(p - 2)}{2}\int e_u^{-1}(\partial_t u) \cdot du(\nabla e_u) e_u^\frac{p - 2}{2}\chi^2 \notag\\
& \pheq - \frac{p(p - 2)}{4}\int e_u^{-1}|\nabla e_u|^2e_u^\frac{p - 2}{2}\chi^2 -(p - 2)E \notag\\
&=: (I) + (II) \notag\\
& \pheq + (III) + (IV) \notag\\
& \pheq + (V) + (VI)\,.
\end{align}
$(I)$ gets handled similarly to the treatment of the original Bochner formula. By applying Lemma \ref{lem:Hessian-of-F}, we have 
\begin{align}
(I) & = -p\int e_u^{\frac{p - 2}{2}}(\nabla^2 F_K)_u(\nabla u,\nabla u)\chi^2 \notag\\
& \le C \int e_u^{\frac{p - 2}{2}} |(\nabla F_K)_u| |\nabla u|^2 \chi^2 \notag\\
& \le \frac{p}{8} \int e_u^{\frac{p - 2}{2}} |(\nabla F_K)_u|^2 \chi^2 
+ C(p) \int_{\supp \chi} e_u^{\frac{p + 2}{2}}\,,
\end{align}
where the first term is absorbable, and the second appears on the r.h.s of \eqref{eq:eps-reg-W22-eq}. 

$(II)$ and $(III)$ are handled using integration by parts in time and space, respectively:
\begin{align}
(II) & = -\int (\partial_t e_u^{p/2})\chi^2 = \int e_u^{p/2}\partial_t\chi^2 - \int_{B_1} e_{u_0}^{p/2}\chi_0^2 \leq C(\|\partial_t \chi^2\|_\infty)\int_{\supp \chi}e_u^{p/2}, \notag\\
(III) & = \int (\Delta e_u^{p/2})\chi^2 = \int e_u^{p/2} \Delta \chi^2 \le C(\|\Delta \chi^2\|_\infty) \int_{\supp \chi}e_u^{p/2}\,.
\end{align}
The error term $(VI)$ (i.e., $E$) can be absorbed by $(V)$:
\begin{align}
\label{eq:(V)+(VI)}
(V) + (VI) &\leq - \frac{p(p - 2)}{4}\int e_u^{-1}|\nabla e_u|^2e_u^\frac{p - 2}{2}\chi^2 + p \int e^{\frac{p - 2}{2}}|\nabla e_u||\nabla \chi|\chi \notag\\
&\leq - \frac{p(p - 2)}{4}\int e_u^{-1}|\nabla e_u|^2e_u^\frac{p - 2}{2}\chi^2 + \frac{p(p - 2)}{8}\int e_u^{-1}|\nabla e_u|^2e_u^\frac{p - 2}{2}\chi^2 \notag\\
& \pheq + C(p)\int_{\supp \chi} e_u^{p/2}|\nabla\chi|^2 \notag\\
&\leq - \frac{p(p - 2)}{8}\int e_u^{-1}|\nabla e_u|^2e_u^\frac{p - 2}{2}\chi^2 + C(p, \|\nabla \chi\|_\infty)\int_{\supp \chi} e_u^{p/2}\,,
\end{align}
using Peter-Paul for the final inequality.

Finally, $(IV)$ is handled analogously:
\begin{align}
(IV) &= -\frac{p(p - 2)}{2}\int e_u^{-1}(\partial_t u) \cdot du(\nabla e_u)e_u^\frac{p - 2}{2}\chi^2 \notag\\
&\leq \frac{p(p - 2)}{16}\int e_u^{-2}|du(\nabla e_u)|^2e_u^\frac{p - 2}{2}\chi^2 + C(p)\int_{\supp \chi} e_u^\frac{p - 2}{2}|\partial_t u|^2 \notag\\
&\leq \frac{p(p - 2)}{16}\int e_u^{-1}|\nabla e_u|^2e_u^\frac{p - 2}{2}\chi^2 + C(p, n)\int_{P_2} e_u^{p/2}\,,
\end{align} 
where the first term on the r.h.s can be absorbed into the negative term of the $(V) + (VI)$ upper bound, and the second term on the r.h.s was bounded using Lemma \ref{lem:loc-energy-ineq}.
\end{proof}

\section{Weak solutions of the $\delta$-regularized flow}
\label{sec:K-Limit}

\subsection{Convergence of the $(\delta,K)$-flow in $\eps$-regularity regions}

Given a collection $\{u_K\}_{K > 0}$ of solutions of the $(\delta,K)$-flow, we will now verify that -- in the $\eps$-smallness regime -- we can extract a subsequence that converges in an appropriate sense to a solution $u$ of the $(p,\delta)$-flow \eqref{eq:reg-flow}. 

\begin{lemma}
\label{lem:eps-reg-HMF-distr}
For $\delta > 0$ fixed, let $u_K$ be smooth solutions to \eqref{eq:flow-delta-K}${}_K$ (respectively) in $P_1$, and assume that $|\nabla u_K|_{\delta,K} \leq M$ in $P_1$ with a uniform constant $M$. Then, up to taking a subsequence, $u_K \to u$ in $C^0(P_{1/2})$ and $\nabla u_K \to \nabla u$ in $L^q(P_{1/2})$ for all $q < \infty$, where $u$ solves the $\delta$-flow \eqref{eq:reg-flow} weakly:
\begin{equation}
|\nabla u|_\delta^{p - 2}\pl_t u -  \dv(|\nabla u|_\delta^{p-2} \nabla u) = |\nabla u|_\delta^{p - 2}A_u(\nabla u, \nabla u)\,.
\end{equation}
\end{lemma}

\begin{proof}
Let's begin by listing the necessary estimates that we get from our previous work: by Corollary \ref{cor:eps-reg-Lip}, Lemma \ref{lem:loc-energy-ineq}, and Theorem \ref{th:eps-reg-W22} respectively, we have the ($\delta$-dependent) bounds
\begin{equation}
[u_K]_{\Lip(P_{1/2})}\,, \
\| \pl_t u_K \|_{L^2(P_{1/2})}\,, \
\|(\nabla F_K)_{u_K}\|_{L^2(P_{1/2})} \,, \
\|\nabla^2u_K\|_{L^2(P_{1/2})} \
\leq C(n,M,\delta)\,.
\end{equation}
In fact, by the assumed upper bound on the energy density, for $K \geq r_0^{-1}M$ we have $\dist^2(u_K, \cN) \leq K^{-2}M^2$, so that by the compactness of $\cN$ we get a uniform bound $\| u_K \|_{C^0} \le C$, and we can replace the bound on $\|(\nabla F_K)_{u_K}\|_{L^2(P_{1/2})}$ with 
\begin{equation}
\|K^2\dist(u_K,\cN)\|_{L^2(P_{1/2})} \le C(n,M,\delta)\,.
\end{equation}

After passing to a subsequence, we can assume by the above estimates that $\partial_t u_K \wto \partial_t u$ and $\nabla^2 u_K \wto \nabla^2 u$ weakly in $L^2$, $\nabla u_K \stackrel{*}{\wto} \nabla u$ in $L^\infty$ (with $\|\nabla u\|_\infty \leq M$), and $-\frac{1}{2}(\nabla F_K)_{u_K}\wto \eta$ weakly in $L^2$ for some $\eta \in L^2(P_{1/2},\R^d)$, all on $P_{1/2}$. Moreover, since $u \in L^2((-1/4,0), W^{2,2}(B_{1/2}))$ and $\pl_t u \in L^2((-1/4,0), L^2(B_{1/2}))$, an application of the Aubin--Lions lemma leads to strong $L^2$ convergence $\nabla u_K \to \nabla u$ \cite[Lem.~7.7]{Rou13}. By the Arzela-Ascoli theorem, $u_K \to u$ in $C^0$. Finally, due to the estimate $\dist^2(u_K, \cN) \leq K^{-2}M^2$ and pointwise convergence, the limiting map $u$ takes values in $\cN$. 

\medskip

Note that by interpolation, strong $L^2$ convergence $\nabla u_K \to \nabla u$ together with weak-* convergence in $L^\infty$ implies strong convergence in any $L^q$ with $1 \le q < \infty$. Similarly, since 
$K^2 F(u_K) \le M$ pointwise and $\|K^4 F(u_K)\|_{L^1(P_1)} \le C$ (the latter follows from the $L^2$ bound on $\nabla F_K$), then 
\begin{equation}
\int_{P_1} |K^2 F(u_K)|^q 
= K^{-2} \int_{P_1} K^4 F(u_K) \cdot |K^2 F(u_K)|^{q-1} 
\le K^{-2} \cdot C \cdot M^{q-1} \to 0\,,
\end{equation}
which means that $K^2 F(u_K) \to 0$ in any $L^q$. The conclusion of these two observations is that 
\begin{equation}
|\nabla u_K|_{\delta, K} \to |\nabla u|_{\delta} 
\quad \text{in } L^q, \quad 1 \le q < \infty\,,
\end{equation}
simply because $|\nabla v|_{\delta, K}^2 = \delta^2 + |\nabla v|^2 + K^2 F(v)$ and $|\nabla v|_{\delta}^2 = \delta^2 + |\nabla v|^2$ by definition. 

\medskip

Let's now verify that $u$ solves \eqref{eq:reg-flow} weakly by showing that the convergence of the $u_K$ is enough to pass to the limit in \eqref{eq:reg-flow}. Take a test function $\phi \in C_c^{\infty}(P_1)$, and let us examine the distributional convergence of each term in turn:

\smallskip

\textsc{Term 1:} 
$|\nabla u_K|^{p - 2}_{\delta, K}\partial_t u_K\to |\nabla u|^{p - 2}_{\delta}\partial_t u$ 

By previous remarks, $|\nabla u_K|^{p - 2}_{\delta, K} \to |\nabla u|^{p - 2}_{\delta}$ strongly in any $L^q$, in particular in $L^2$. Since additionally $\partial_t u_K \wto \partial_t u$ weakly in $L^2$, the product converges weakly in $L^1$ to the product of the limits. In consequence, 
\begin{equation}
\int \phi |\nabla u_K|^{p - 2}_{\delta, K} \partial_t u_K 
\to \int \phi |\nabla u|^{p - 2}_{\delta} \partial_t u\,.
\end{equation}

\smallskip

\textsc{Term 2:} 
$\dv(|\nabla u_K|_{\delta,K}^{p - 2} \nabla u_K)\to \dv(|\nabla u|^{p - 2}_{\delta}\nabla u)$ 

With this term, we can integrate by parts:
\begin{equation}
\int \phi \dv(|\nabla u_K|_{\delta,K}^{p - 2} \nabla u_K) = \int |\nabla u_K|_{\delta,K}^{p - 2}\nabla \phi \cdot \nabla u_K \,.
\end{equation}
Now we are able apply a similar reasoning: the convergence of $|\nabla u_K|_{\delta,K}$ and $\nabla u_K$ to their respective counterparts is strong in any $L^q$, so the same is true for the product $|\nabla u_K|_{\delta,K}^{p - 2} \nabla u_K$, and we infer that 
\begin{align}
&\int \phi \dv(|\nabla u_K|_{\delta,K}^{p - 2} \nabla u_K) = \int |\nabla u_K|_{\delta,K}^{p - 2}\nabla \phi \cdot \nabla u_K \notag\\
&\xrightarrow{K \to \infty} \int |\nabla u|_{\delta}^{p - 2}\nabla \phi \cdot \nabla u = \int \phi\dv(|\nabla u|_{\delta}^{p - 2} \nabla u)\,,
\end{align}
meaning that $\dv(|\nabla u_K|_{\delta,K}^{p - 2} \nabla u_K)\to \dv(|\nabla u|^{p - 2}_{\delta}\nabla u)$ in the weak sense. 

\smallskip

\textsc{Term 3:} 
$-\frac{1}{2}|\nabla u_K|^{p - 2}_{\delta,K}(\nabla F_K)_{u_K} \to \tilde{\eta} \in u^{-1}(T^\perp\cN)$ 

Recall the $L^2$-weak convergence asserted at the beginning of the proof: $-\frac{1}{2}(\nabla F_K)_{u_K} \wto \eta$. As in the case of the first term, we thus have $-\frac{1}{2}|\nabla u|_{\delta,K}^{p - 2}(\nabla F_K)_{u_K} \wto \tilde{\eta} := |\nabla u|^{p - 2}_{\delta} \eta \in L^2$, weakly in $L^1$. 
It remains to show that $\eta$ takes values in the pullback normal bundle $u^{-1}(T^\perp \cN)$. To this end, take any $\psi\in C^0(P_{1/2}, u^{-1}(T\cN))$. Assuming again that $K \geq r_0^{-1}M$ is large enough, we then have 
\begin{align}
\left|\int \psi\cdot \eta \right| 
&= \lim_{K\to\infty}\frac{1}{2}\left|\int \psi\cdot(\nabla F_K)_{u_K} \right| \notag\\
& \leq \limsup_{K \to \infty} \int K^2\dist(u_K,\cN)|\psi\cdot (\nabla \dist_{\cN})_{u_K}| \notag\\
& \leq \limsup_{K \to \infty}\left(\int K^4\dist^2(u_K,\cN)\right)^{1/2}\left(\int|\psi\cdot (\nabla \dist_{\cN})_{u_K}|^2\right)^{1/2} \notag\\
& \ale_{n,\delta,M} \limsup_{K \to \infty} \| \psi \cdot (\nabla \dist_{\cN})_{u_K} \|_{L^\infty(P_{1/2})}\,. 
\end{align}
Now, since $\psi(z)$ lies in $T_{u(z)} \cN$ and $(\nabla \dist_{\cN})_{u_K(z)}$ lies in $T^\perp_{u_K(z)} \cN$, their inner product is bounded by 
\begin{align}
|\psi(z) \cdot (\nabla \dist_{\cN})_{u_K(z)}|&= |\psi(z) \cdot (P_{T_{u(z)}\cN} - P_{T_{u_K(z)}\cN} + P_{T_{u_K(z)}\cN})(\nabla \dist_{\cN})_{u_K(z)}| \notag\\
&= |\psi(z) \cdot (P_{T_{u(z)}\cN} - P_{T_{u_K(z)}\cN})(\nabla \dist_{\cN})_{u_K(z)}| \notag\\
&\ale_{\cN} |u_K(z)-u(z)| \cdot |\psi(z)|\,.
\end{align}
Thanks to the uniform convergence $u_K \to u$ in $C^0(P_{1/2})$, we conclude that the limit above is $0$.

Having handled each term, we can proceed to the distributional equation for $u$. In particular:
\begin{align}
0 &= \int \phi(|\nabla u_K|^{p - 2}_{\delta,K}\partial_t u_K - \dv(|\nabla u_K|^{p - 2}_{\delta,K}\nabla u_K) + \frac{1}{2}|\nabla u_K|^{p - 2}_{\delta,K}(\nabla F_K)_{u_K}) \notag\\
&\to \int\phi(|\nabla u|^{p - 2}_{\delta}\partial_t u - \dv(|\nabla u|^{p - 2}_{\delta}\nabla u) - \tilde{\eta})\,,
\end{align}
so that $|\nabla u|^{p - 2}_{\delta} \pl_t u - \dv(|\nabla u|^{p - 2}_{\delta}\nabla u) = \tilde{\eta}$ distributionally. Finally, since $u$ maps into $\cN$, we necessarily have that $|\nabla u|^{p - 2}_{\delta}\partial_t u$ lies in the tangent space $T_u\cN$, while 
\begin{equation}
P_{T^\perp_u\cN} \left[ - \dv(|\nabla u|^{p - 2}_{\delta} \nabla u) \right] 
= |\nabla u|^{p - 2}_{\delta} P_{T^\perp_u\cN} \left[ - \Delta u \right]
= |\nabla u|_\delta^{p - 2}A_u(\nabla u, \nabla u)\,.
\end{equation}
Since $\tilde{\eta} \in L^2(u^{-1}(T^\perp \cN))$, we conclude that it has to coincide with the above, and thus $u$ solves $\delta$-flow distributionally.
\end{proof}

\subsection{Minkowski content estimates on the concentration set}

According to the $\eps$-regularity theorem and the previous subsection, the sequence $u_K$ converges to a solution whenever the scale-invariant energy is $\eps$-small. This only leaves out points $z$ such that, whatever $s > 0$ we choose, the scale-invariant energy $\Phi^{\delta,K}_z(s; u_{K})$ remains large. We shall now verify that this remaining set -- the concentration set -- is small. 

As we shall see, this analysis does not depend on the PDE satisfied by $u_K$. Hence, in the following we do not assume that $u_K$ solves a PDE. Instead, we only assume that $u_K$ is a sequence of functions for which the quantity $\Phi^{\delta,K}_z(s; u_{K})$ from Lemma \ref{th:monotonicity-formula} is non-decreasing in $s > 0$ and that the corresponding initial maps $u_{K,0}$ satisfy a uniform bound $\int_{\R^n} (|\nabla u_{K,0}|^2+F_K(u_{K,0}))^{p/2} \le V$. Note that this is an $L^p$ bound on $|\nabla u_{K,0}|_{0,K}$, which we consider here because the $L^p$ norm of $|\nabla u_{K,0}|_{\delta,K}$ is always infinite for $\delta > 0$. 

If the maps $u_{K,0}$ take values in $\cN$, as it happens in our intended application, this energy assumption reduces to a~uniform bound on $\| \nabla u_{K,0} \|_{L^p(\R^n)}$. Moreover, the proof will only make use of energy bounds with Gaussian weights: 
\begin{equation}
s^{p/2} H_s |\nabla u_{K,0}|^p_{\delta,K} (x) \le V(s,x) 
\qquad \text{for all } s > 0, \ x \in \R^n\,,
\end{equation}
in which the constant may depend on everything except $K$. Such bounds are easily seen to follow from $\int_{\R^n} |\nabla u_{K,0}|^p_{0,K} \le V$. On the other hand, they imply that $\Phi^{\delta,K}_{(t,x)}(s; u_{K}) \le V(t,x)$ by monotonicity. 

\begin{definition}
For such a sequence -- satisfying monotonicity and with uniformly bounded energy -- let us define the concentration set $\cS$, which plays the role of the singular set: 
\begin{equation}
\cS := \left\{ z \in (0,\infty) \times \R^n \mid \forall_{s > 0} \, \liminf_{K \to \infty} \Phi^{\delta,K}_z(s; u_{K}) \ge \eps_0 \right\}\,.
\end{equation}
where $\eps_0 = \eps_0(n,\cN,p) > 0$ is the same as in Theorem \ref{th:eps-regularity}.
\end{definition}

\begin{lemma}
The concentration set $\cS$ is closed.
\end{lemma}

\begin{proof}
We will show that the complement of $\cS$ is open. To this end, fix a point $z = (t,x) \notin \cS$ and choose some small $s > 0$ for which 
\begin{equation}
\gamma := \liminf_{K \to \infty} \Phi^{\delta,K}_{z}(s; u_K) < \eps_0\,.
\end{equation}
Then fix a subsequence (still indexed by $K$ for simplicity) which realizes the limit inferior above. For $z' = (t',x')$ close to $z$ let us choose the unique $s' > 0$ satisfying $t'-s' = t-s$. Denoting the heat kernel by $\rho$, note that 
\begin{equation}
|\rho_{s'}(x',y) - \rho_s(x,y)| \le \omega(z') \rho_{2s}(x,y)
\quad \text{uniformly in } y\,,
\end{equation}
as in the proof of Theorem \ref{th:eps-regularity}, where $\omega(z') > 0$ is a quantity that tends to zero as $z' \to z$. Comparing the monotone quantity $\Phi^{\delta,K}$ at the two points, we have 
\begin{align}
\left| (s/s')^{p/2} \Phi^{\delta,K}_{z'}(s'; u_K) - \Phi^{\delta,K}_{z}(s; u_K) \right| 
& = s^{p/2} \left| H_{s'} |\nabla u_{t'-s'}|^p_{\delta,K}(x') - H_{s} |\nabla u_{t-s}|^p_{\delta,K}(x) \right| \notag\\
& \le s^{p/2} \int |\nabla u_{t-s}(y)|_{\delta,K}^p |\rho_{s'}(x',y) - \rho_s(x,y)| \dd y \notag\\
& \le \omega(z') \cdot s^{p/2} \int |\nabla u_{t-s}(y)|_{\delta,K}^p \rho_{2s}(x,y) \dd y \notag\\
& = 2^{-p/2} \omega(z') \cdot \Phi^{\delta,K}_{(t+s,x)}(2s; u_{K})
\ale V(t+s,x) \omega(z')\,.
\end{align}
The calculation above used the equality $t'-s' = t-s$ and exploited monotonicity of $\Phi^{\delta,K}$ together with Gaussian bounds on $u_{K,0}$. The resulting upper bound may depend on $z$, but this is not relevant here. Recalling that $s' \to s$ as $z' \to z$, we can choose a neighborhood $U$ of $z$ (independently of $K$) so that $z' \in U$ implies $\Phi^{\delta,K}_{z'}(s'; u_K) \le \frac{\gamma+\eps_0}{2}$, and in particular $z' \notin \cS$. 

This shows that the complement of $\cS$ is open, which means that $\cS$ is closed. 
\end{proof}

As a step towards $(n+2-p)$-dimensional Minkowski content estimates, it will be useful to compare the local scale-invariant energy $r^{p-n} \int_{B_r} |\nabla u|_{\delta,K}^p$ and the monotone scale-invariant energy $\Phi^{\delta,K}(s)$. A one-sided estimate follows simply by restricting the integral defining $\Phi^{\delta,K}(s)$ to a ball. In the opposite direction, one has to deal with the Gaussian tail, which results in a small contribution at a larger scale. 

\begin{lemma}
\label{lem:local-global-energy-comparison}
Choose a point $z_0 = (t_0,x_0)$ satisfying $0 < T_0 \le t_0 \le T_1$. Let $0 < r < 1$ and $0 < \eta < 1$ be such that $s := (\eta r)^2$ is smaller than $t_0$. Then 
\begin{equation}
\label{eq:local-global-energy-comparison}
\Phi^{\delta,K}_{(t_0,x_0)}(s) \le (\eta r)^{p-n} \int_{B_r(x_0)} |\nabla u_{t_0-s}|_{\delta,K}^p \dd y 
+ C(n) \eta^{-n} e^{-\frac{1}{8\eta^2}} \Phi_{(t_0+r^2,x_0)}(s+r^2)\,,
\end{equation}
and in particular, one can choose $\eta(n,p,\eps_0,T_0,T_1,V) > 0$ so that 
\begin{equation}
\Phi^{\delta,K}_{z_0}(s) \le C(n,p,\eps_0,T_0,T_1,V) \cdot r^{p-n} \int_{B_r(x_0)} |\nabla u_{t_0-s}|_{\delta,K}^p \dd y 
+ \frac{\eps_0}{4}\,.
\end{equation}
\end{lemma}

\begin{proof}
First note that the heat kernel $\rho$ satisfies 
\begin{equation}
\rho_s(y) \le (\eta r)^{-n} \text{ in } B_r\,, \qquad
\rho_s(y) \ale_n \eta^{-n} e^{-\frac{1}{8\eta^2}} \rho_{s+r^2}(y) \text{ outside } B_r.
\end{equation}
The first of these inequalities is a direct consequence of the formula $\rho_s(y) = (4 \pi s)^{-n/2} \exp(-|y|^2/4s)$, and the second relies on the following comparison: 
\begin{align}
\frac{\rho_s(y)}{\rho_{s+r^2}(y)} 
& = \left( \frac{s}{s+r^2} \right)^{-n/2} \exp \left( -\frac{|y|^2}{4s} + \frac{|y|^2}{4(s+r^2)} \right) \notag\\
& = \left( 1+\eta^{-2} \right)^{n/2} \exp \left( -\frac{|y|^2}{4r^2} \cdot \frac{1}{\eta^2(1+\eta^2)} \right) \notag\\
& \le 2^{n/2} \eta^{-n} e^{-\frac{1}{8\eta^2}} 
& \text{for } |y| \ge r, \ \eta \le 1\,.
\end{align}
As a consequence of these two inequalities, we have 
\begin{equation}
H_s |\nabla u_{t_0-s}|_{\delta,K}^p (x_0)
\le (\eta r)^{-n} \int_{B_r(x_0)} |\nabla u_{t_0-s}|_{\delta,K}^p \dd y 
+ C(n) \eta^{-n} e^{-\frac{1}{8\eta^2}} H_{s+r^2} |\nabla u_{t_0-s}|_{\delta,K}^p (x_0)\,.
\end{equation}
Multiplying both sides by $s^{p/2} = (\eta r)^p$, we clearly obtain \eqref{eq:local-global-energy-comparison} (one needs to bound $s^{p/2}$ by $(s+r^2)^{p/2}$ in the last term).

For the second statement, note that $\Phi_{(t_0+r^2,x_0)}(s+r^2)$ is bounded by $V(t_0+r^2,x_0)$ due to monotonicity and Gaussian bounds on $u_{K,0}$. If $t_0 \in [T_0,T_1]$, one can express this bound in terms of $V$ and $T_0,T_1$ only. Now, since $\eta^{-n} e^{-\frac{1}{8\eta^2}} \to 0$ as $\eta \to 0$, it is possible to choose $\eta > 0$ small enough so that the last term does not exceed $\eps_0/4$. Of course, the contribution of $\eta^{p-n}$ in the first term of the r.h.s then becomes a large constant.
\end{proof}

We will say that the set $E \subset \R^{n + 1}$ has (parabolic) $k$-Minkowski content bounded by $C$ (or simply that $\Min_k(E) \le M$), if for small $r > 0$, any disjoint collection of $N$ (parabolic) balls $Q_r(z_i)$ centered at $z_i \in E$ has $N \cdot r^k \le M$, where $M$ does not depend on $r$. Note that this is essentially stronger than a Hausdorff measure bound $\cH^{k}(E) \le M$. 

It is easy to see that this condition implies (at least for compact $E$) a bound on the parabolic $r$-neighborhood of $E$: $|Q_r(E)| \ale_n M r^{n + 2 - k}$. Let us briefly recall that 
\begin{equation}
Q_r(E) := \bigcup_{z \in E} Q_r(z) = \left\{ z \mid d_P(z,E) < r \right\}\,.
\end{equation}
To see why the bound holds, one covers $Q_r(E)$ by balls $Q_r(x)$ (with $x \in E$), then chooses a Vitali subcovering: $Q_r(E) \subseteq \bigcup Q_{5r}(x_i)$ with $Q_r(x_i)$ disjoint. By the previous condition, the number $N$ of balls in this subcovering has to satisfy $N \cdot r^k \le M$, and in consequence 
\begin{equation}
|Q_r(E)| 
\le \left| \bigcup Q_{5r}(x_i) \right|
\le \sum_i |Q_{5r}(x_i)|
\ale_n N \cdot r^{n + 2} 
\le M r^{n + 2 - k}\,.
\end{equation}
In fact, these two definitions are equivalent (up to a constant) -- the opposite inequality follows from the fact that any disjoint collection of balls $Q_r(x_i)$ centered at $x_i \in E$ lies inside $Q_r(E)$, and so $2 \omega_n N r^{n + 2} \le |Q_r(E)|$. 

\begin{lemma}
\label{lem:n+2-p-content}
The concentration set $\cS$ has locally bounded $(n + 2 - p)$-Minkowski content.
\end{lemma}

\begin{remark}
\label{rmk:Misawa-rmk}
This estimate can be viewed as an improvement of the Hausdorff measure estimates of \cite[Thm.~5]{Mis19} (albeit for a different flow). This last paper considers the $\cH^{m}$ measure with respect to the metric $d_{\gamma_0}((t,x),(s,y)) := |t - s|^{1/\gamma_0} + |x - y|$ for any $2 < \gamma_0 < p$, while we consider the $\cH^{m + 2 - p}$ measure w.r.t the standard parabolic metric $d_P((t,x),(s,y)) = \max(|t - s|^{1/2}, |x - y|) \approx d_2((t,x),(s,y))$. Let us note that by a covering argument -- using that $Q_r^{d_P}$ can be covered by $\approx r^{2 - \gamma_0}$ translated copies of $Q_r^{d_{\gamma_0}}$ -- the latter measure controls the former, even for $\gamma_0 = 2,p$. Thus, Lemma \ref{lem:n+2-p-content} achieves the optimal local finiteness of $\cH^m$-measure w.r.t the metric $d_p$ for the concentration set as discussed in \cite[Rmk. Pg.~6]{Mis19}, sharpening the estimate of \cite[Thm.~5]{Mis19}.
\end{remark}

\begin{proof}
We will show that for any $0 < T_0 < T_1$ and any compact $L$, the part of the singular set 
\begin{equation}
\cS' := \cS \cap \left( [T_0,T_1] \times L \right)
\end{equation}
has finite $(n + 2 - p)$-Minkowski content. To this end, consider a disjoint collection of $N$ parabolic balls $Q_r(z_i)$ centered at $z_i = (t_i,x_i) \in \cS'$. With $\eps_0 > 0$ as in $\eps$-regularity (Theorem \ref{th:eps-regularity}) we choose small $\eta > 0$ as in Lemma \ref{lem:local-global-energy-comparison} and fix $s := (\eta r)^2$. For each $i$, by definition we have $\liminf_{K \to \infty} \Phi^{\delta,K}_{z_i}(s; u_K) \ge \eps_0$. Since there are only finitely many $i$'s, we may choose $K$ large enough so that 
\begin{equation}
\Phi^{\delta,K}_{z_i}(s; u_K) \ge \frac{\eps_0}{2} 
\quad \text{for } i=1,2,\ldots,N.
\end{equation}
From now on, we will fix such $K$ and denote $u_K$ by $u$ to avoid clutter. Applying Lemma \ref{lem:local-global-energy-comparison}, we can get rid of the Gaussian tail and obtain a similar lower bound for the local scale-invariant energy: 
\begin{align}
\frac{\eps_0}{2}
& \le \Phi^{\delta,K}_{z_i}(s; u_K) \notag\\
& \le C(n,p,\cN,T_0,T_1,V) \cdot r^{p-n} \int_{B_r(x_i)} |\nabla u_{t_0-s}|_{\delta,K}^p \dd y 
+ \frac{\eps_0}{4}\,, \\
\frac{\eps_0}{4}
& \ale_{n,p,\cN,T_0,T_1,V} r^{p-n} \int_{B_r(x_i)} |\nabla u_{t_0-s}|_{\delta,K}^p \dd y\,. 
\end{align}
The implicit constant in the final inequality depends on Gaussian bounds on $u_{K,0}$; since our attention is restricted to $\cS'$, the constant is uniform for all considered points. 

It is desirable to obtain such a bound also for an integral over spacetime. To this end, we repeat the above reasoning for each $s' \in [s/2,s]$ and the same value of $r$ (which means that $\eta$ has to vary slightly), obtaining a lower bound for $|\nabla u_{t_0-s'}|_{\delta,K}^p$. Averaging all these inequalities, we have 
\begin{equation}
\eps_0 
\ale_{n,p,\cN,T_0,V} r^{p-n} \fint_{s/2}^{s} \int_{B_r(x_i)} |\nabla u_{t_0-s'}|_{\delta,K}^p \dd y \dd s'
\ale_\eta r^{p-n-2} \int_{[t_0-s,t_0-s/2] \times B_r(x_i)} |\nabla u|_{\delta,K}^p  \dd y \dd t\,.
\end{equation}
The parabolic ball $Q_r(z_i)$ contains $[t_0-s,t_0-s/2] \times B_r(x_i)$, so the integral over the ball is even larger. Finally, we can sum up the contribution of all balls and obtain 
\begin{equation}
N \eps_0 \cdot r^{n+2-p} \ale_{n,p,\cN,T_0,V} \int_{\bigcup Q_r(z_i)} |\nabla u|_{\delta,K}^p  \dd y \dd t 
\le \int_{[T_0,T_1]} \int_{L} |\nabla u|_{\delta,K}^p  \dd y \dd t\,. 
\end{equation}
For each $t \in [T_0,T_1]$ the inner integral can be bounded by the scale-invariant energy and thus also by the initial energy. In result, the integral is bounded uniformly in $K$. This finishes the proof of the local Minkowski content bound $\Min_{n+2-p}(\cS') \le C(n,p,\cN,V,T_0,T_1,L)$. 
\end{proof}

\subsection{Convergence of the $(\delta,K)$-flow over the whole domain}

We have already showed that the sequence of approximate solutions $u_K$ converges (up to a subsequence) to a solution of the $\delta$-flow \eqref{eq:reg-flow} in $\eps$-regularity regions. It is not hard to establish some convergence of weak type on the whole domain, but due to non-linearity of the problem, it is not clear whether the limit satisfies the equation, even in a weak sense. A~natural approach is to test \eqref{eq:reg-flow} with an additional cut-off function which vanishes near $\cS$. Using content estimates for (a~compact portion of) $\cS$, we can indeed show that cutting off is not too expensive. In result, we obtain a weak formulation of \eqref{eq:reg-flow} on the whole domain. 

\begin{theorem}
\label{th:HMF-distr}
Let $u_K$ be solutions to $\eqref{eq:flow-delta-K}_K$ for $\delta > 0$ fixed, and assume that their initial maps $u_{K,0}$ satisfy a~uniform bound $\int_{\R^n} (|\nabla u_{K,0}|^2 + F_K(u_{K,0}))^{p/2} \le V$.
Then up to taking a subsequence, $u_K \wto u$ weakly in $W_\Loc^{1,2}$ (strongly in $L^2_\Loc$) and $\nabla u_K \wto \nabla u$ weakly in $L^p_\Loc$, where $u$ solves the $\delta$-flow \ref{eq:reg-flow} distributionally.
\end{theorem}

\begin{remark}
The main application of this Lemma is when the $u_K$ have the same initial data $u_0$ taking values in $\cN$. In this case, the energy assumptions simplify to $\nabla u_0 \in L^p(\R^n)$. Moreover, it is natural to ask whether the limiting map $u$ preserves the initial data $u_0$ -- we address this question in Lemma \ref{lem:preserving-initial-data}. 
\end{remark}

\begin{proof}
Since the conclusions of the lemma are local, we may as well work on a parabolic cylinder $P_R(z_0) \subset (0,\infty) \times \R^{n}$, where $R > 0$ and $z_0 = (t_0,x_0) \in [4R^2, \infty) \times \R^{n}$. In this region, we have by the (rescaled) local energy inequality (Lemma \ref{lem:loc-energy-ineq})
\begin{align}
\int_{P_R(z_0)}|\partial_t u_K|^2 
&\leq \delta^{2 - p}\int_{P_R(z_0)}|\nabla u_K|_{\delta,K}^{p - 2}|\partial_t u_K|^2 \notag\\
&\ale_{\delta, n,p} R^{- 2} \int_{P_{2R}(z_0)} |\nabla u_K|^p_{\delta,K}\,,
\end{align}
and by the monotonicity formula (Theorem \ref{th:monotonicity-formula})
\begin{align}
\left(\int_{P_{2R}(z_0)}|\nabla u_K|^2 \right)^{p/2} 
&\ale_{R,p,n} \int_{P_{2R}(z_0)} |\nabla u_K|_\delta^p \notag\\
&= \int_{t_0-4R^2}^{t_0} \int_{B_{2R}(x_0)} |\nabla u_K|^p_\delta\,dyds \notag\\
&\ale_n \int_{t_0-4R^2}^{t_0}R^{n - p}\Phi^{\delta,K}_{(t_0 + R^2, x)}(u_K, t_0 + R^2- s) ds \notag\\
&\leq R^{n - p}\Phi^{\delta,K}_{(t_0 + R^2, x)}(u_K,t_0 + R^2)\int_{t_0-4R^2}^{t_0} ds \notag\\
&\ale_{R,n,p} V(t_0 + R^2, x)\,.
\end{align}
By Rellich-Kondrachov compactness, we can pick a subsequence of the $K \in (0,1)$ (not relabeled) so that $(\partial_t u_K,\nabla u_K) \wto (\partial u,\nabla u)$ weakly in $L^2(P_R(z_0))$, $u_K \to u$ strongly in $L^2(P_R(z_0))$. Moreover, since $\|\nabla u_K\|_{L^p(\Omega_R)} \ale_{p,n,R,t_0} V$ by the previous inequalities, we can assume that $\nabla u_K \wto \nabla u$ weakly in $L^p(P_R(z_0))$.

If $z = (t,x)$ is not in the energy concentration set $\cS$, there is a scale $s > 0$ and a subsequence $K_i \to \infty$ such that $\Phi_{(t,x)}^{\delta,K_i}(s_z,u_{K_i}) < \eps_0$. In particular, by the $\eps$-regularity theorem (Theorem \ref{th:eps-regularity}) we have a uniform bound $|\nabla u_{K_i}|_{\delta,K_i} \le M$ on some small parabolic cylinder $P_{\gamma s}(z)$ ($M,\gamma$ may depend on $z$ but not on $i$). By Lemma \ref{lem:eps-reg-for-later-time}, a possibly weaker uniform bound holds on some neighborhood of $z$. Hence, Lemma \ref{lem:eps-reg-HMF-distr} implies that $u$ solves the $\delta$-flow \eqref{eq:reg-flow} weakly on this neighborhood (note that the limit obtained in Lemma \ref{lem:eps-reg-HMF-distr} coincides with $u$). Since $z \notin \cS$ was arbitrary, we conclude that $u$ solves the $\delta$-flow \eqref{eq:reg-flow} weakly on $P_R(z_0) \setminus \cS$. 

The last step is to make sure that the equation is satisfied on the whole domain, not only away from $\cS$. To this end, fix a test function $\phi \in C^\infty_c((0,\infty)\times \R^n, \R^d)$, let $L$ be its compact support and let $\cS' := \cS \cap L$. For a given $r > 0$, we will construct a cutoff function $\chi_r \in C^\infty_c(P_R(z_0))$ that satisfies $0 \leq \chi_r \leq 1$, $\chi_r \equiv 0$ on $Q_r(\cS')$, $\chi_r \equiv 1$ outside $Q_{2r}(\cS')$ and $|\nabla \chi_r| \ale r^{-1}$. 

Such a cutoff function can be constructed as follows. We first let $\eta_1(x) := \dist(x,(Q_{2r}(\cS'))^c)$, which is the parabolic distance to the complement of $Q_{2r}(\cS')$. This satisfies $0 \leq \eta_1 \leq 2r$, $\eta_1 \equiv 0$ outside $Q_{2r}(\cS')$ and $|\nabla \eta_1| \le 1$. Taking $\eta_2 := \max(\eta_1,r)$, we ensure that also $\eta_2 \equiv r$ in $Q_r(\cS')$. It remains to rescale and switch the values $0$ and $1$, i.e. consider $\eta_3 := 1-\eta_2/r$, and take a mollification of $\eta_3$ to obtain the desired~$\chi_r$. 

Now, since the \eqref{eq:reg-flow} is satisfied away from $\cS$, $\phi \chi_r$ can be used as a test function, leading to 
\begin{align}
0 &= \int \phi\chi_r(|\nabla u|_\delta^{p - 2}\partial_t u - \dv(|\nabla u|_\delta^{p - 2} \nabla u) - |\nabla u|_\delta^{p - 2}A_u(\nabla u,\nabla u)) \notag\\
&= \int \chi_r(\phi|\nabla u|_\delta^{p - 2}\partial_t u + |\nabla u|_\delta^{p - 2} \nabla u \cdot \nabla \phi - \phi|\nabla u|_\delta^{p - 2}A_u(\nabla u,\nabla u)) \notag\\
& \pheq + \int \phi|\nabla u|_\delta^{p - 2}\nabla u\cdot \nabla \chi_r\,.
\end{align}
We now consider the limit $r \to 0$. By dominated convergence, the first line of the r.h.s tends to 
\eqref{eq:reg-flow} paired with the test function $\phi$, so it remains to show that the second line tends to $0$. This is where the content estimate shines:
\begin{align}
\left|\int \phi|\nabla u|_\delta^{p - 2}\nabla u\cdot \nabla \chi_r\right| 
&\ale_{\phi} \int |\nabla u|_\delta^{p - 2}|\nabla u||\nabla \chi_r| \notag\\
&\ale r^{-1}\int_{\supp\nabla \chi_r} |\nabla u|_\delta^{p - 1} \notag\\
&\leq r^{-1}|Q_{2r}(\cS')|^\frac{1}{p}\left(\int_{Q_{2r}(\cS')}|\nabla u|_\delta^p\right)^\frac{p - 1}{p} \notag\\
&\ale_{n,p,\cN, V,t_0,R} \left(\int_{Q_{2r}(\cS')}|\nabla u|_\delta^p\right)^\frac{p - 1}{p}\,,
\end{align}
where the last step employs the estimate $|Q_{2r}(\cS')| \ale_{n,p,\cN, V,t_0,R} r^p$ from Lemma \ref{lem:n+2-p-content}. Using this estimate again, we infer that $|Q_{2r}(\cS')| \to 0$ as $r \to 0$, and the r.h.s. tends to zero by absolute continuity of the Lebesgue integral. This finishes the proof that $u$ solves the $\delta$-flow \eqref{eq:reg-flow} weakly on $P_R(z_0)$. 
\end{proof}

The main interest in Theorem \ref{th:HMF-distr} is that it provides a construction for a solution of \eqref{eq:reg-flow} with given smooth initial data $u_0 \colon \R^n \to \cN$ (assuming that the approximating equations can be solved, see Remark \ref{rmk:flow-existence}). To this end, one needs to check that the initial data is preserved in the limiting process: 

\begin{lemma}
\label{lem:preserving-initial-data}
Fix $\delta > 0$ and a smooth initial map $u_0 \colon \R^n \to \cN$ with finite energy $\int_{\R^n} |\nabla u_0|^p < \infty$. Assume that $u_K \colon [0,\infty) \times \R^n \to \R^d$ are smooth solutions of $\eqref{eq:flow-delta-K}_K$. Given a convergent subsequence $u_k \to u$ as in Theorem \ref{th:HMF-distr}, the limiting map $u$ has the same initial data $u_0$ in the sense that $u_t \xrightarrow{t \to 0} u_0$ in $L^2_\Loc$. 
\end{lemma}

\begin{proof}
Since the conclusion is local, let us fix a point $x_0 \in \R^n$. We shall use the local energy inequality on the cylinder $(0,1) \times B_1(x_0)$, but in a slightly modified form. We repeat the argument from Lemma \ref{lem:loc-energy-ineq} with a time-independent cut-off $\mathbbm{1}_{B_1(x_0)} \le \chi \le \mathbbm{1}_{B_2(x_0)}$. In such case, the counterpart of equation \eqref{eq:local-energy-ineq-main-eq} involves an additional integral over the initial time slice $t = 0$: 
\begin{gather}
\frac 1p \int |\nabla u^K_1|^p_{\delta,K} \chi
- \frac 1p \int |\nabla u^K_0|^p_{\delta,K} \chi
+ \frac 12 \int_{(0,1) \times \R^n} |\nabla u^K|^{p-2}_{\delta,K} |\pl_t u^K|^2 \chi
\le C(p,n) \int_{(0,1) \times B_2(x_0)} |\nabla u^K|^{p}_{\delta,K}\,,
\end{gather}
which leads to 
\begin{equation}
\int_{(0,1) \times B_1(x_0)} |\partial_t u^K|^2 
\ale_{n,p,\delta} 
\int_{B_2(x_0)} |\nabla u_0|^p_{\delta,K}
+ \int_{(0,1) \times B_2(x_0)} |\nabla u^K|^{p}_{\delta,K}\,. 
\end{equation}
The first term on the right does not depend on $K$. Using the monotonicity formula as in the previous lemma, the second term can be also bounded in terms of $u_0$. In consequence, $\pl_t u^K$ is uniformly bounded in $L^2(P_1(z_0))$, where $z_0 = (1,x_0)$. It follows from H\"older's inequality that 
\begin{equation}
\int_{B_1(x_0)} |u^K_t-u^K_s|^2 \dd x \le |t-s| \int_{[t,s] \times B_1(x_0)} |\pl_t u^K|^2 \dd t \dd x\,, 
\end{equation}
which means that $u^K$ is uniformly $\frac 12$-H\"older continuous as a map $t \mapsto u^K_t$ from $(0,1)$ to $L^2(B_1(x_0))$. In particular, $\| u^K_t - u_0 \|_{L^2} \ale_{n,p,\delta,u_0} t^{1/2}$ with a uniform constant. 

As in the previous lemma, an application of Rellich--Kondrachov leads to a subsequence $u^K$ convergent strongly in $L^2(P_1(z_0))$. If we choose a further subsequence, we have $u^K_t \to u_t$ in $L^2(B_1(x_0))$ for all $t$ in a set $E \subseteq (0,1)$ of full measure. It now follows from previous observations and the triangle inequality that $\| u_t - u_0 \|_{L^2} \ale_{n,p,\delta,u_0} t^{1/2}$ for $t \in E$. This completes the proof, as one can redefine $u$ on the remaining null measure set to obtain the desired representative. 
\end{proof}

\section{Applications and improvements} \label{sec:Apps}

\subsection{Target manifold with nonpositive sectional curvature}

Since Eells and Sampson's seminal work \cite[p.~147]{EelSam64}, it is well understood that the classical harmonic map flow is better behaved if the target manifold is non-positively curved. More precisely, in that case Bochner's formula implies that the energy density $|\nabla u|^2$ is a subsolution of the heat equation, and in consequence it stays bounded. Since a possible blow-up of the energy density is the only obstacle to extending the solution in time, one can conclude that global-in-time smooth solutions exist. 

In this section we show that the $\delta$-regularized $p$-harmonic map flow \eqref{eq:reg-flow} inherits these good geometric properties of its classical counterpart, at least partially. In our case, Bochner's formula (Lemma \ref{lem:delta-Bochner-explicit}) still ensures that the energy density $|\nabla u|_\delta^p$ is a subsolution of a parabolic equation, and hence remains bounded. We record this observation as a corollary: 

\begin{corollary}
\label{cor:neg-curv-target-reg}
Let $\cM$, $\cN$ be two closed Riemannian manifolds, and assume that $\cN$ has nonpositive sectional curvature. If $u \colon [0,T) \times \cM \to \cN$ is a smooth solution of the $\delta$-flow \eqref{eq:reg-flow} with an initial map satisfying $|\nabla u_0|_\delta \le M$ everywhere on $\cM$, then 
\begin{equation}
|\nabla u(t,x)|_\delta \le e^{C(\cM,p) \cdot t} \cdot M
\qquad \text{for } x \in \cM, \, 0 < t < T\,,
\end{equation}
in other words, the energy density remains bounded. 
\end{corollary}

Note that this theorem \emph{does not} guarantee the solution can be extended to a larger time interval. This is because the corresponding regularity statement is unknown for general parabolic systems, even in the flat case: an $\R^d$-valued solution of $\pl_t u - |\nabla u|_\delta^{2-p} \dv(|\nabla u|_\delta^{p-2} \nabla u) = 0$ could possibly have a uniform bound on $|\nabla u|$ but not on $|\nabla^2 u|$, thus failing to be extendable, see Remark \ref{rmk:flow-existence}. 

\begin{proof}
For simplicity, let us first consider the case of flat $\cM$. Bochner's formula (Lemma \ref{lem:delta-Bochner-explicit}) then implies 
\begin{equation}
(\pl_t-\cA) |\nabla u|_{\delta}^p
\le p |\nabla u|_{\delta}^{p-2} \sum_{\alpha,\beta} \Rm(\pl_\alpha u, \pl_\beta u, \pl_\beta u, \pl_\alpha u)\,,
\end{equation}
for a certain elliptic operator $\cA$ (which depends on $u$). Since $\cN$ has nonpositive sectional curvature, each term on the right-hand side is nonpositive, which means that $|\nabla u|_{\delta}^p$ is a subsolution of a parabolic equation. It follows from the classical maximum principle that $|\nabla u(t,x)|_\delta \le M$ for $x \in \cM$ and $0 < t < T$. 

The general case requires a minor modification, which is standard (see e.g. \cite[Lem.~5.3.3]{LinWan08} for details). More precisely, changing the order of differentiation leads to an additional term involving the Ricci curvature of $\cM$. This leads to a slightly weaker bound 
\begin{equation}
\label{eq:curved-delta-L-Bochner-formula}
(\pl_t-\cA) |\nabla u|_{\delta}^p
\le C(\cM,p) |\nabla u|_\delta^p\,,
\end{equation}
and the previous argument can be applied to $e^{-C(\cM,p)t} |\nabla u|_\delta^p$, resulting in the final claim. 
\end{proof}

The main application of Eells and Sampson's existence and regularity result for the classical harmonic map flow into non-positively curved targets is to produce a smooth harmonic map that is homotopic to the initial data. If we assume existence and regularity, then we can obtain the analogous result for the $\delta$-flow:

\begin{theorem}
\label{th:nonpositive-curvature-limit}
Let $\cM$, $\cN$ be two closed Riemannian manifolds, and assume that $\cN$ has nonpositive sectional curvature. If $u: [0,\infty) \times \cM \to \cN$ is a smooth solution to the $\delta$-flow \eqref{eq:reg-flow}, then there is a map $u_\infty \in W^{1,\infty}(\cM,\cN)$ that is homotopic to the inital data $u_0$, and that weakly solves the $\delta$-regularized $p$-harmonic map equation
\begin{equation}
\label{eq:delta-harmonic-map-eq}
-\dv(|\nabla u_\infty|^{p - 2}_\delta \nabla u_\infty) 
= |\nabla u_\infty|^{p - 2}_\delta A_{u_\infty}(\nabla u_\infty, \nabla u_\infty)\,.
\end{equation}
\end{theorem}

The idea, as in \cite{EelSam64}, is to obtain $u_\infty$ as a subsequential limit of time slices $u_t$ as $t\to \infty$. In particular, one chooses the subsequence among maps whose time derivative tends to $0$. We note that the global existence and regularity must be assumed here and cannot be proven directly, as in Remark \ref{rmk:flow-existence} and the discussion following the statement of Corollary \ref{cor:neg-curv-target-reg}.

\begin{proof}
The proof uses three main ingredients: the global energy rate of change formula, the Bochner formula, and the parabolic Harnack inequality. First, we will use the following rate of change formula for the $\delta$-energy:
\begin{equation}
\label{eq:closed-M-energy-ident}
E_\delta(u_t) + \int_{[0,t] \times \cM} |\nabla u|^{p - 2}_\delta |\partial_t u_t|^2 = E_\delta(u_0) 
\qquad \text{for } t > 0 \,,
\end{equation}
which follows by the same computation as in Lemma \ref{lem:energy-inequality} (applied to the $\delta$-flow in place of the $(\delta,K)$-flow). Now, we use the Bochner formula \eqref{eq:curved-delta-L-Bochner-formula}, which implies that the nonnegative quantity $v:= e^{-C(\cM,p)t}|\nabla u|^{p}_\delta$ is a subsolution of the parabolic operator $(\partial_t - \cA)$. We may then apply the Harnack inequality, Theorem \ref{th:Parabolic-Harnack}, to conclude that
\begin{equation}
v(t,x) \leq \frac{C(p, \cM)}{R^{n + 2}} \int_{[t - R^2, t] \times \cM} v \dd V \dd t 
\qquad \text{for } R < \text{inj}(x), \ t > R^2\,.
\end{equation}
Note that we may apply Theorem \ref{th:Parabolic-Harnack} by working in normal coordinates, since $R$ is smaller than the injectivity radius. This may change the ellipticity constant of $\cA$ and the density of the volume measure $dV$, but only by a factor $C(\cM)$. 

In particular, we can use \eqref{eq:closed-M-energy-ident} to further estimate the right-hand side, with
\begin{align}
|\nabla u|^{p}_\delta(x,t) = e^{C(\cM,p)t}v(x,t) &\leq \frac{C(\cM,p)}{R^{n + 2}}\int_{t - R^2}^te^{-C(\cM,p)(s - t)} E_\delta(u_s)\,ds \notag\\
& \leq C(p,R,\cM) E_\delta(u_0)\,.
\end{align}
With this uniform-in-time energy density bound in hand, by using the methods of Lemma \ref{lem:eps-reg-HMF-distr}, we may pick a~subsequence $t_i \to \infty$ such that $u_i(t,x) := u(t_i + t,x) \to u_\infty$ in $C^0((0,1) \times \cM)$, where $u$ solves the $\delta$-flow \eqref{eq:reg-flow} weakly. This method of proof also yields the weak convergence $\pl_t u_i \wto \pl_t u_\infty$ in $L^2((0,1) \times \cM)$.

On the other hand, the rate of change formula \eqref{eq:closed-M-energy-ident} with $t \to \infty$ implies that
\begin{equation}
\int_0^1 \int_{\cM} |\pl_t u_i|^2 \dd V \dd t 
\ale_{\delta,p} \int_{[t_i,t_i+1] \times \cM} |\nabla u|^{p - 2}_\delta |\pl_t u|^2 \dd V \dd t 
\longrightarrow 0 
\qquad \text{ as } i \to \infty\,.
\end{equation}
By the weak convergence $\pl_t u_i \wto \pl_t u_\infty$, we conclude that $u_\infty$ is independent of time. For time-independent maps, the $\delta$-flow equation is simply the $\delta$-regularized $p$-harmonic map equation \eqref{eq:delta-harmonic-map-eq} for any time slice.

To see that $u_0$ is homotopic to $u_\infty$, let us first note that $u_0$ is homotopic to $(u_i)_0$ for any $i$, with the homotopy witnessed by the flow $u \in C^0([0,t_i] \times \cM)$. Moreover, because of the strong convergence $u_i \to u_\infty \in C^0((0,1) \times \cM)$, we have strong convergence $(u_i)_0 \to (u_\infty)_0 \in C^0(\cM)$. In particular, $i$ may be chosen large enough so that the line segment $\gamma_x:[0,1]\to\R^d$ connecting the points $(u_i)_0(x),\,(u_\infty)_0(x) \in \cN \subset \R^d$ lies in the tubular neighborhood $B_{r_0}(\cN)$ where the nearest-point projection $\pi_{\cN}: B_{r_0}(\cN) \to \cN$ is continuous. In particular, the map $(s,x)\mapsto \pi_{\cN}(\gamma_x(s))$ is a homotopy between $u_i$ and $u_\infty$. Thus, the maps $u_0 \simeq (u_i)_0 \simeq (u_\infty)_0$ are homotopic.
\end{proof}

\subsection{Small initial energy}

As in Struwe's work in the $p = 2$ case \cite[Sec.~7]{Str88}, one can easily see that the $\delta$-flow with small initial energy is particularly well-behaved -- more precisely, it collapses to a point. Once again, the available existence theory is not sufficient to prove that $u_t \to \mathrm{const.}$ as $t \to \infty$, simply because the existence of a global-in-time solution is not guaranteed, see Remark \ref{rmk:flow-existence}. However, one can record the following simple proposition which assumes a possibly finite time horizon $T$, but the quantitative estimates are not dependent on $T$: 

\begin{proposition}
Let $u \colon [0,T) \times \R^n \to \cN$ be a smooth solution of the $\delta$-flow \eqref{eq:reg-flow} with $2 \le p \le n$ and with the initial map satisfying a global Lipschitz bound $|\nabla u_0| \le L$. Then there is $\eps_1(n,\cN,p,L) > 0$ such that 
\begin{equation}
\int_{\R^n} |\nabla u_0|_\delta^p - \delta^p \le \eps_1
\quad \Longrightarrow \quad 
|\nabla u_t(x)| \le \frac{C}{\sqrt{t}} \quad \text{for } x \in \R^n, \ 0 < t < T.
\end{equation}
\end{proposition}

There are two main differences between the classical case $p = 2$ and general $p \ge 2$. The first one is again the ability to conclude global-in-time existence from energy density bounds, the second difference only appears for the $\delta$-regularized flow. In the non-regularized case $\delta = 0$ one easily sees that the small energy assumption implies $\Phi_{(t,x)}(t; u) \le \eps_0$ at every point $(t,x)$. Rescaling (by $\sqrt{t}$) to the unit scale, one can apply $\eps$-regularity, which rescaling back, leads to $|\nabla u| \ale 1/\sqrt{t}$. If $\delta > 0$, one always has $\Phi_{(t,x)}(t; u) \ge t^{p/2} \delta^p$, which cannot be small for large $t > 0$. 

\begin{proof}[Sketch of proof]
For reasons outlined above, we consider the reduced monotone quantity 
\begin{equation}
\ov{\Phi}^{\delta}_{(t,x)} (s; u) := s^{p/2} H_s \left( |\nabla u_{t - s}|_{\delta}^p - \delta^p \right) (x)\,.
\end{equation}
It is still nonnegative and scale invariant: 
\begin{equation}
u^{(\lambda)}(t,x) = u(\lambda^2 t, \lambda x)
\quad \Longrightarrow \quad 
\ov{\Phi}_{(0,0)}^{\delta}(\lambda^2 s; u) = \ov{\Phi}_{(0,0)}^{\lambda \delta} \big( s; u^{(\lambda)} \big)\,.
\end{equation}
Moreover, a monotonicity formula holds: $\pl_s \ov{\Phi}^{\delta}_{(t,x)} (s; u) \ge 0$. These observations are easy consequences of the corresponding properties of $\Phi$. Inspecting the proof of Bochner's formula (Theorem \ref{lem:delta-Bochner-explicit}, especially \eqref{eq:delta-Bochner-first-calculation}), we see that 
\begin{equation}
(\pl_t - \cA) |\nabla u|^2 
\le 2\ps{\nabla u}{\nabla( A_u(\nabla u, \nabla u) )} 
+ \frac{p-2}{2} |\nabla u|^{-2} |\nabla |\nabla u|^2|^2\,,
\end{equation}
the only difference stemming from $e_u^{-1} \le |\nabla u|^{-2}$. In consequence, Bochner's formula holds also directly for $|\nabla u|^2$: $(\pl_t-\cA) |\nabla u|^p \ale |\nabla u|^{p+2}$. Finally, the proof of $\eps$-regularity (Theorem \ref{th:eps-regularity}) trivially carries over to the case of the $\delta$-flow. It is easily verified that also a version for the reduced monotone quantity holds: 
\begin{equation}
\ov{\Phi}^{\delta}_{(0,0)}(1; u) \le \eps_0
\quad \Longrightarrow \quad 
|\nabla u| \le C 
\text{ in } P_\gamma\,.
\end{equation}
The only non-trivial change in the proof involves the pointwise bound $|\nabla u|^p \le |\nabla u|_\delta^p - \delta^p$. 

\medskip

With these amendments, the proof from the classical case carries over. The main point of the proof is the observation that the scale-invariant energy of $u_0$ is $\eps_0$ small at every scale. To see this, first use the bound $|\nabla u_0|^p_\delta \le (L^2+\delta^2)^{p/2}$ to infer 
\begin{equation}
r^p \int_{\R^n} \left( |\nabla u_0(x)|^p_\delta - \delta^p \right) \rho_{r^2}(x,y) \dd x 
\le r^p (L^2+\delta^2)^{p/2} \le \eps_0\,,
\end{equation}
if only $0 < r < r_0$ and $r_0(L,\delta)$ is small. At larger scales $r \ge r_0$, 
\begin{equation}
r^p \int_{\R^n} \left( |\nabla u_0(x)|^p_\delta - \delta^p \right) \rho_{r^2}(x,y) \dd x 
\le C(n) r^p r^{-n} \eps_1 
\le \eps_0\,,
\end{equation}
if only we choose $\eps_1 \le C(n)^{-1} \eps_0 r_0^{n-p}$. Thus, we have verified that the reduced monotone quantity $\ov{\Phi}$ is $\eps_0$-small: $\ov{\Phi}_{(t,x)}(t; u) \le \eps_0$ at every point $(t,x)$ with $0 < t < T$. Fixing such a point and applying $\eps$-regularity to the map $v(t',x') := u(t+tt', x + \sqrt{t} x')$, we infer that $|\nabla v(0,0)| \le C$, and rescaling back, that $|\nabla u(t,x)| \le \frac{C}{\sqrt{t}}$. 
\end{proof}

\appendix

\section{Formulas for the $\delta$-regularized flow}

The main focus of this paper was the analysis of solutions for the approximating problem, the $(\delta,K)$-flow \eqref{eq:flow-delta-K}, with special emphasis on the (lack of) dependence on $K$. However, smooth solutions of the $\delta$-flow enjoy many of the same properties, and the presented proofs easily carry over to this case. For the reader's convenience, we provide brief accounts of the two main ingredients: the monotonicity formula (counterpart of Theorem \ref{th:monotonicity-formula}) and Bochner's formula (counterpart of Theorem \ref{th:Bochner}). 

\begin{lemma}[Monotonicity formula]
\label{lem:delta-monotonicity-formula}
Let $u$ be a smooth solution of the $\delta$-flow \eqref{eq:reg-flow} and denote $\Phi^{\delta}_{(t,x)} (s; u) := s^{p/2} H_s |\nabla u_{t - s}|_{\delta}^p (x)$. 
Then $s \mapsto \Phi^{\delta,K}_{(t,x)}(s)$ is a non-decreasing function. More precisely, its derivative is exactly 
\begin{equation}
\pl_s \Phi_{(t,x)}^{\delta}(s) 
= \frac{p}{2} s^{\frac{p-4}{2}} H_s \left( |\nabla u_{t-s}|_{\delta}^{p-2} \left( \frac 12 |2 s \pl_t u_{t-s} - (y-x) \nabla u_{t-s}|^2 
+ s \delta^2 \right) \right) (x)\,.
\end{equation}
\end{lemma}

\begin{proof}[Sketch of proof]
The proof follows by exactly the same computation as in Theorem \ref{th:monotonicity-formula} with the simplification $K = 0$; the only difference lies in the equation satisfied by $u$. Recall that by the previous computation we obtain 
\begin{gather}
\pl_\lambda \big|_{\lambda=1} \Phi_{(0,0)}^{\delta}(\lambda^2; u)
= p \int_{\R^n} |\nabla u|_{\delta}^{p-2} \cdot \delta^2 \cdot \rho_1 \dd x \notag\\
+ p \int_{\R^n} \left( -|\nabla u|_{\delta}^{p-2} \nabla u \cdot \left( -\frac x2 \right) 
- \dv \left( |\nabla u|_{\delta}^{p-2} \nabla u \right) \right)
\cdot (-2 \pl_t u + x \nabla u) \rho_1 \dd x\,,
\end{gather}
where, to avoid clutter, we omitted the subscript in $u_{-1}$ indicating that the map is to be evaluated at time $t=-1$. Here, the $p$-Laplace term $\dv \left( |\nabla u|_{\delta,K}^{p-2} \nabla u \right)$ is not exactly equal to $|\nabla u|_{\delta}^{p-2} \pl_t u$, but the difference between these two is orthogonal to $\cN$, whereas the term $-2 \pl_t u + x \nabla u$ is tangent to $\cN$. Thus, the second integral simplifies to 
\begin{equation}
\int_{\R^n} \left( -|\nabla u|_{\delta}^{p-2} \nabla u \cdot \left( -\frac x2 \right) 
- |\nabla u|_{\delta}^{p-2} \pl_t u \right) 
\cdot (-2 \pl_t u + x \nabla u) \rho_1 \dd x\,,
\end{equation}
which is indeed a perfect square. After some simplifications, we arrive at the desired form of $\pl_s \Phi_{(t,x)}^{\delta}(s)$, at least for $s = 1$. The general case follows by rescaling, as in Theorem \ref{th:monotonicity-formula}. 
\end{proof}

\begin{lemma}
\label{lem:delta-Bochner-explicit}
If $u$ is a smooth solution of the $\delta$-flow \eqref{eq:reg-flow}, and $\cA$ is the elliptic operator 
\begin{equation}
\cA \psi := \Delta \psi + (p - 2) |\nabla u|_{\delta}^{-2} \nabla^2 \psi(\nabla u, \nabla u)\,,
\end{equation}
then the $p$-energy density $|\nabla u|_{\delta}^p$ satisfies the following Bochner-type formula: 
\begin{align}
\label{eq:delta-L-Bochner-formula}
(\pl_t-\cA) |\nabla u|_{\delta}^p
\le p |\nabla u|_{\delta}^{p-2} \sum_{\alpha,\beta} \Rm(\pl_\alpha u, \pl_\beta u, \pl_\beta u, \pl_\alpha u)\,.
\end{align}
\end{lemma}

\begin{proof}
We repeat the computation in Lemma \ref{lem:Bochner-explicit}, with $-\frac 12 \nabla F_K(u)$ replaced by $A_u(\nabla u, \nabla u)$. First, we compute the heat operator applied to $e_u$:
\begin{align}
\label{eq:delta-Bochner-first-calculation}
(\partial_t - \Delta) e_u 
&= 2\ps{\nabla u}{\nabla( A_u(\nabla u, \nabla u) )} - 2|\nabla^2 u|^2 - (p-2) e_u^{-2}|\nabla u\cdot \nabla e_u|^2 \notag\\
& \pheq + \frac{p-2}{2} e_u^{-1} |\nabla e_u|^2 + (p-2) e_u^{-1}(\nabla^2 e_u)(\nabla u,\nabla u)\,.
\end{align}
As before, the second-order term $(p-2) e_u^{-1}(\nabla^2 e_u)(\nabla u,\nabla u)$ is used to form the operator $\cA$. Then, we consider $e_u^{p/2}$ to get rid of the term $\frac{p-2}{2} e_u^{-1} |\nabla e_u|^2$, resulting in 
\begin{align}
(\pl_t-\cA) e_u^{p/2}
& = p e_u^\frac{p-2}{2} \left( - \frac{p(p-2)}{4} e_u^{-2} |\nabla u \cdot \nabla e_u|^2 
+ \ps{\nabla u}{\nabla( A_u(\nabla u, \nabla u) )} - |\nabla^2 u|^2 \right)\,.
\end{align}
The first term is nonpositive, so we only need to interpret the remaining two in terms of curvature. This is a~classical computation which we now recall. 

Observe that each derivative $\pl_\alpha u$ is orthogonal to $A_u(\nabla u, \nabla u)$, which implies that 
\begin{equation}
\ps{\nabla u}{\nabla( A_u(\nabla u, \nabla u) )}
= - \ps{\Delta u}{A_u(\nabla u, \nabla u)}
\end{equation}
by differentiating $\ps{\pl_\alpha u}{A_u(\nabla u, \nabla u)} \equiv 0$. 
Then recall that the orthogonal part of each derivative $\pl_{\alpha\beta} u$ is $-A_u(\pl_\alpha u, \pl_\beta u)$, so that the above equals
\begin{equation}
\ps{\sum_{\alpha} A_u(\pl_\alpha u, \pl_\alpha u)}{\sum_{\beta} A_u(\pl_\beta u, \pl_\beta u)}
= \sum_{\alpha,\beta} \ps{A_u(\pl_\alpha u, \pl_\alpha u)}{A_u(\pl_\beta u, \pl_\beta u)}\,.
\end{equation}
Similarly, we can decompose the Hessian term -- ignoring the tangential second-order derivatives, we obtain 
\begin{equation}
- |\nabla^2 u|^2 
\le - \sum_{\alpha,\beta} |(\pl_{\alpha\beta}u)^{\perp}|^2 
= - \sum_{\alpha,\beta} \ps{A_u(\pl_\alpha u, \pl_\beta u)}{A_u(\pl_\alpha u, \pl_\beta u)}\,.
\end{equation}
Now the claim follows by Gauss' equation 
\begin{equation}
\Rm(v_1,v_2,v_3,v_4) = \ps{A(v_1,v_4)}{A(v_2,v_3)} - \ps{A(v_1,v_3)}{A(v_2,v_4)}\,.
\end{equation}
\end{proof}

\bibliographystyle{aomalpha}
\bibliography{HMN_epsreg}

\end{document}